\documentclass[journal,twoside,web]{ieeecolor}
\input{header.tex}
\fulltrue
\title{Differentiable-by-design Nonlinear Optimization \\ for Model Predictive Control}
\author{Riccardo Zuliani, Efe C. Balta, and John Lygeros
\thanks{Corresponding author: R. Zuliani. This work was supported as a part of NCCR Automation, a National Centre of Competence in Research, funded by the Swiss National Science Foundation (grant number 51NF40\_225155). All authors are with the Automatic Control Laboratory (IfA), ETH Z\"urich, 8092 Z\"urich, Switzerland \texttt{\small$\{$rzuliani,lygeros$\}$@ethz.ch}. E. C. Balta is also with inspire AG, 8005 Z\"urich, Switzerland. \texttt{\small efe.balta@inspire.ch} }}
\begin{document}
%
\maketitle
\begin{abstract}
Nonlinear optimization-based control policies, such as those arising in nonlinear Model Predictive Control, have seen remarkable success in recent years. These policies require solving computationally demanding nonlinear optimization programs online at each time-step. The resulting solution map, viewed as a function of the measured state of the system and design parameters, may not be differentiable, which poses significant challenges if the control policy is embedded in a gradient-based policy optimization scheme. 
We propose a principled way to regularize the nonlinear optimization problem \rvv{around the nominal solution}, obtaining a surrogate derivative even if the \rvv{solution map of the} original problem is not differentiable. 
The surrogate problem is differentiable by design and its solution coincides with the solution of the unregularized problem. 
We demonstrate the effectiveness of our approach in a free-final-time optimal control problem and a receding-horizon nonlinear MPC example.
\end{abstract}
\section{Introduction}
\IEEEPARstart{T}{hanks} to the rapid advancement in real-time computational capabilities, optimization-based control techniques are now \rvv{widely recognized as effective across} both industry and academia \cite{Schwenzer2021}. These approaches compute the control inputs by solving an optimization problem, making it possible to directly include cost specifications and constraints within the control design.
The most prominent example of this is \rvv{optimal control, whose most widely used form is} Model Predictive Control (MPC), where the objective is to optimize the predicted future state and input trajectory of a system to minimize a predefined cost under input and process constraints \rvv{in receding horizon} \cite{rawlings2017model}. To predict the system's future evolution accurately, a model of the system dynamics is required. For systems with nonlinear dynamics, this leads to a parametric nonlinear programming problem (NLP), as the optimization typically depends on the system state and design parameters (e.g., the terminal cost). State-of-the-art solvers, such as the open-source IPOPT \cite{wachter2006implementation}, can solve these NLPs efficiently; however, their solutions may be nonsmooth or even discontinuous as a function of the optimization parameters.

One approach to enhance the performance of optimization-based controllers is policy optimization \cite{zuliani2023bp,agrawal2020learning}, which involves a gradient-based tuning process aimed at identifying the optimal controller for a given task. \rvv{This is achieved by defining a performance metric that evaluates the closed-loop behavior of the system, parameterizing the controller, and computing the gradient of this closed-loop cost with respect to the parameters.}

In recent years, policy optimization has been successfully applied to MPC controllers, yielding impressive results \cite{zuliani2023bp,agrawal2020learning,zuliani2024closed,gros2019data}. However, the general lack of differentiability \rvv{of the solution map of} nonlinear programs presents a significant challenge for policy optimization when using nonlinear MPC controllers, as policy optimization requires both the controller and the system dynamics to be sufficiently smooth.
In this paper, we propose a principled method to regularize an NLP to ensure that its solution map is differentiable by design, while maintaining close proximity to the unregularized solution. 
In contrast to existing approaches \cite{gros2019data,pirnay2012optimal,andersson2018sensitivity}, which focus on differentiating NLPs under restrictive assumptions (e.g., SSOSC, LICQ, SCS, see \cref{section:diff}), our method applies to any NLP with twice continuously differentiable objectives and constraints, with minimal assumptions on the problem's minimizer. Specifically, our contributions are as follows:
\begin{enumerate}
    \item we propose a regularization for a generic NLP to guarantee that the regularized problem has a differentiable solution map which is locally equivalent to the solution of the original problem;
    \item we show that, if the solution of the original problem is differentiable, then the two derivatives \rvv{coincide};
    \item we propose an algorithm to obtain the derivative of the regularized problem.
\end{enumerate}

\subsubsection*{Related work}

\rvv{
The solution map of a nonlinear program is differentiable under the strong second-order sufficient conditions of optimality (SSOSC), the linear independence constraint qualification (LICQ), and strict complementarity slackness (SCS), as first established by Fiacco in his seminal paper \cite{Fiacco1964}. At points where SCS does not hold, such as when the active set changes, the solution map is merely directionally differentiable \cite{Jittorntrum1984}. Subsequent research relaxed the LICQ assumption, though this comes at the cost of requiring a stronger second-order sufficient condition \cite{Kojima1980,Kyparisis1990}. 

These theoretical foundations have been practically applied in various forms for over half a century. 
For example, \cite{pirnay2012optimal} utilizes them to derive sensitivities in IPOPT, while \cite{frey2025differentiable} recently integrated them into acados. 
\cite{frey2025differentiable}, in particular, bypasses the SCS requirement by differentiating a barrier subproblem where the barrier parameter is strictly positive. 
This smoothing strategy dates back to Fiacco \cite{Fiacco1964} for primal-only formulations and was later extended to primal-dual formulations by \cite{pirnay2012optimal}. 
Under SSOSC, LICQ, and SCS, the continuous derivative of this barrier subproblem converges to the derivative of the original problem as the barrier parameter approaches $0$. 
Whenever SCS does not hold, one can still obtain a derivative for the relaxed problem; however, it remains unclear what this quantity converges to as the barrier parameter is taken to $0$ (and whether such limit even exists).

Sensitivity analysis is also a cornerstone of path-following methods \cite{Allgower1990}, which aim to track the solution of an optimization problem as underlying parameters change. The standard approach in the literature is the Euler-Newton predictor-corrector method \cite{Allgower1990,Dontchev2013,Dinh2012}, alternating a sensitivity-based Taylor approximation with a Newton-style correction. While these methods possess rigorous tracking guarantees, their underlying analysis relies on strict assumptions regarding the solution path. Classical continuation \cite{Allgower1990} requires a full-rank KKT matrix, while modern adaptations for parametric optimization \cite{Dinh2012} and variational inequalities \cite{Dontchev2013} assume Robinson's strong regularity. Because strong regularity for standard nonlinear programs is equivalent to the combination of SSOSC and LICQ, these methods generally struggle at points where regularity breaks. Recently, \cite{LiaoMcPherson2020} relaxed these requirements using the concept of semiderivatives, relying only on a ``critical face condition'' or alternatively an easily verifiable second-order sufficient condition. However, the framework of \cite{LiaoMcPherson2020} is restricted to polyhedral inequality constraints.

Gradient-based policy optimization has gained significant traction in recent years, with differentiation techniques being heavily utilized to solve problems involving both nonlinear \cite{drgovna2024learning,oshin2023differentiable,gros2019data,Zavala2009} and convex \cite{zuliani2023bp,zuliani2024closed,amos2017optnet,amos2018differentiable,agrawal2020learning,tao2024difftune} optimization-based policies. In these problems, parameters update dynamically, and regularity conditions frequently fail, causing standard derivative-based solvers to struggle.

Ultimately, the existing literature predominantly focuses on identifying theoretical conditions (e.g., SSOSC, LICQ, or structural constraints) that ensure some notion of differentiability.
Our work takes a fundamentally different direction. 
Rather than imposing stringent conditions, we investigate how to regularize the problem itself to yield a differentiable mapping.
Furthermore, we show how the derivative of the regularized problem relates to the derivative of the original problem whenever the latter exists, providing a reliable sensitivity framework even when classical assumptions fail.}

\subsubsection*{Notation}
We use $\operatorname*{col}(a,b)$ to denote the vector or matrix obtained by vertically stacking $a$ and $b$. We use $\operatorname*{diag}(a,b)$ to denote the matrix whose diagonal entries are the elements of $a$ and $b$.
For $a<b$, we use $\mathbb{Z}_{[a,b]}$ to denote the set of integers contained between $a$ and $b$.
%
\section{Problem statement}
Consider the parameterized nonlinear program (NLP)
\begin{align}
\begin{split}
\operatorname*{minimize}_x & \quad f_\theta(x)\\
\text{subject to} & \quad g_\theta(x) \leq 0\\
& \quad h_\theta(x) = 0,
\end{split} \tag{$P_1(\theta)$} \label{eq:P1}
\end{align}
where $\rvv{\theta \in \R^{n_\theta}}$ is a parameter, $f_\theta:\R^{n_x}\to\R$, $g_\theta:\R^{n_x}\to\R^{n_\text{in}}$, $h_\theta:\R^{n_x}\to\R^{n_\text{eq}}$ are twice continuously differentiable jointly in $\theta$ and $x$. The Lagrangian of \ref{eq:P1} is
\begin{align}
\bar{\mathcal{L}}_\theta(x,\lambda,\nu) = f_\theta(x) + \sum_{i=1}^{n_\text{in}} \lambda_i g_{\theta,i}(x) + \sum_{j=1}^{n_\text{eq}} \nu_j h_{\theta,j}(x), \label{eq:P1_lagrangian}
\end{align}
where $\lambda\in\R^{n_\text{in}}$ and $\nu\in\R^{n_\text{eq}}$ are the multipliers associated to the inequality and the equality constraints, respectively.
We denote with $\rvv{X(\theta) \subseteq \R^{n_x}}$ the set of local minimizers of \ref{eq:P1} for a given $\theta$ (where $X(\theta)=\emptyset$ if \ref{eq:P1} is infeasible).

Under appropriate constraint qualification conditions, if $x$ is a local minimizer of \ref{eq:P1}, then there must exist a pair of Lagrange multipliers $\lambda$, $\nu$ satisfying the KKT conditions
\begin{align}
\begin{split}
\nabla_x \bar{\mathcal{L}}_\theta(x,\lambda,\nu) & = 0\\
g_\theta(x) &\leq 0,\\
h_\theta(x) & = 0,\\
\lambda_i g_{\theta,i}(x) & = 0,~~ \forall i\in\Z_{[1,n_\text{in}]},\\
\lambda_i &\geq 0,~~ \forall i\in\Z_{[1,n_\text{in}]}.
\end{split}\label{eq:KKT}
\end{align}
We use $\rvv{\Phi(\theta) \subseteq \R^{n_\phi}}$, where $n_\phi=n_x+n_\text{in}+n_\text{eq}$, to denote the set of all primal-dual variables $\phi=(x,\lambda,\nu)\in\R^{n_\phi}$ of \ref{eq:P1} that verify \cref{eq:KKT} with $x\in X(\theta)$ for a given $\theta$, where $\Phi(\theta)=\emptyset$ if $X(\theta)=\emptyset$ or if no Lagrange multipliers exist. 

We use $\rvv{\mathcal{I}(x), \bar{\mathcal{I}}(x) \subseteq \Z_{[1,n_\text{in}]}}$ to denote the set of indices $i$ for which $g_{\theta,i}(x)=0$, and $g_{\theta,i}(x) \neq 0$, respectively, and additionally $g_{\theta,\mathcal{I}} = \operatorname*{col}(\{ g_{\theta,i} \}_{i\in \mathcal{I}(x)})$, and similarly for $\bar{\mathcal{I}}(x)$.
\rvv{We use $\mathcal{I}_0(x,\lambda)$ and $\mathcal{I}_+(x,\lambda)$ to denote the set of weakly active and strongly active constraints, respectively, that is, those subsets of $\mathcal{I}(x)$ containing $i\in\Z_{[1,n_\text{in}]}$ for which $\lambda_i=0$ and $\lambda_i>0$, respectively.}

Given a value of $\theta$, under appropriate conditions that we introduce in \cref{section:diff}, the solution set $X(\theta)$ shrinks locally to a singleton, $X(\theta)=\{ x(\theta) \}$. Our objective is to compute the Jacobian of the solution map $x(\theta)$ for a given $\theta$. Whenever this is not possible (for example if $X(\theta)$ is not a singleton), then our objective becomes formulating a surrogate optimization problem by introducing regularization terms in \ref{eq:P1} to ensure that 1) the solution map of the surrogate problem is differentiable and its derivative can be efficiently computed, 2) the solution of the surrogate problem is equal to the one of the original problem at $\theta$, and 3) the surrogate derivative can be used to construct a local first-order approximation of $x(\theta)$. 
%
\section{Preliminaries on derivatives of NLPs}\label{section:diff}
Before presenting our main contributions, we review some fundamental concepts about differentiability of NLPs.

Problem \probone may have infinitely many solutions. The following condition guarantees that a vector $x$ is a strict local minimizer of \probone \cite[Proposition 3.3.2]{bertsekas1997nonlinear}.
\begin{definition}
A primal-dual pair $\phi$ satisfies the \rvv{\emph{second-order sufficient conditions of optimality} (SOSC)} if $\phi\in \Phi(\theta)$ and $y^{\top} \nabla_{xx}^2 \bar{\mathcal{L}}_\theta(\phi) y >0$ for all $y\in \R^{n_x} \setminus \{ 0 \}$ satisfying $\nabla_x h_{\theta}(x)y=0$, $\nabla_x g_{\theta,i}(x)y=0$ \rvv{for all $i\in \mathcal{I}_+(x,\lambda)$, and $\nabla_x g_{\theta,i}(x)y\leq0$ for all $i\in \mathcal{I}_0(x,\lambda)$}.
\rvv{If the condition on $y$ involving $\mathcal{I}_0(x,\lambda)$ is removed, then one obtains the \emph{strong second-order sufficient conditions of optimality} (SSOSC).}
\end{definition}
Similar to $X(\theta)$, the set of primal-dual solutions $\Phi(\theta)$, if not empty, may contain infinitely many points even if $X(\theta)$ is a singleton. The following constraint qualification ensures uniqueness of the multiplier vector and, combined with SSOSC and SCS, guarantees that $\Phi(\theta)$ is single-valued in a neighborhood of $\theta$ \cite[Proposition 3.1.1]{bertsekas1997nonlinear}.
\begin{definition}
A minimizer $x\in X(\theta)$ of \probone satisfies the linear independence constraint qualification (LICQ) if the matrix $\operatorname*{col}(\nabla_x g_{\theta,\mathcal{I}}(x),\nabla_x h_{\theta}(x))$ is full row rank.
\end{definition}
To ensure differentiability, we require one last qualification.
\begin{definition}
A primal-dual pair $\phi$ satisfies the \emph{strict complementarity slackness} (SCS) condition if $\lambda_i>0$ for all $i\in \mathcal{I}(x)$.
\end{definition}
\begin{lemma}\label{lemma:differentiability}
If $\phi\in \Phi(\theta)$ satisfies the SSOSC, LICQ, and SCS, then there exists a neighborhood $N$ of $\theta$ and a continuously differentiable function $\rvv{\varphi}:N \to \R^{n_\phi}$, such that $\rvv{\varphi}(\theta)=\phi$, $\rvv{\varphi}(\tilde{\theta}) \in \Phi(\tilde{\theta})$ satisfies the SSOSC, LICQ, and SCS for all $\tilde{\theta}\in N$; moreover $\nabla_\theta x(\theta)$ is given by the first $n_x$ rows of the unique $V$ satisfying
\begin{align}
A_\theta(\phi)V=b_\theta(\phi), \label{eq:derivative}
\end{align}
with
\begin{align}
\begin{split}
A_\theta(\phi)&=\begin{bmatrix}
\nabla_{xx}^2 \bar{\mathcal{L}}_\theta(\phi) & \nabla_x g_{\theta,\mathcal{I}}(x)^{\top} & \nabla_x h_\theta(x)^{\top} \\
\nabla_x g_{\theta,\mathcal{I}}(x) & 0 & 0 \\
\nabla_x h_\theta(x) & 0 & 0
\end{bmatrix},\\
b_\theta(\phi)&= \rvv{-}\operatorname*{col}(%
\nabla^2_{\theta x} \bar{\mathcal{L}}_\theta(\phi),%
\nabla_\theta g_{\theta,\mathcal{I}}(x),%
\nabla_\theta h_\theta(x)).
\end{split}\label{eq:A_B_KKT_system}
\end{align}
\end{lemma}
\begin{proof}
It follows from \cite[Corollary 2.3.1]{jittorntrum1978sequential} after removing the rows of $\nabla_x g_\theta$ associated to inactive constraints and dividing both the right and the left hand side by the inequality multipliers $\lambda_i$ (since $\lambda_i>0$ for all active constraints).
\end{proof}
Guaranteeing satisfaction of SSOSC, LICQ, and SCS may be challenging for generic NLPs, which may lead to issues with the differentiability and uniqueness of the solution map. We now present a simple example showcasing that, even for small convex optimization problems, the solution map may not be differentiable.
\rvv{The example showcases the strength of our proposed regularization in a simple-to-understand edge-case. Not every problem with a nonunique solution can be dealt with this way.}
\begin{example}\label{example:1}
Consider the following quadratic program (QP)
\begin{align}
\begin{split}
\operatorname*{minimize}_{x\in\R^3}& \quad \frac{\alpha}{2}x_1^2 + x_2 + x_3\\
\text{subject to}& \quad x_1+x_2+x_3=0,
\end{split}\label{eq:QP_example}
\end{align}
where $\alpha >0$ is a parameter. The primal-dual solution set, obtained by solving the KKT system explicitly, is
\begin{align*}
\Phi(\alpha) = \{ (x,\nu) : x_1=1/\alpha, x_2+x_3=-1/\alpha, \nu=-1 \},
\end{align*}
which is not a singleton for any $\alpha>0$. Let $x(\alpha_0)$ be a minimizer of \cref{eq:QP_example} with $\alpha=\alpha_0$. If the value of $\alpha$ changes from $\alpha_0$ to $\alpha_1$, we can quickly find a minimizer of the QP using the approach in \cite[Equations (16.5), (16.6)]{nocedal1999numerical} by solving the following system of equations
\begin{align}
\begin{split}
\begin{bmatrix}
G(\alpha_1) & A^{\top} \\
A & 0
\end{bmatrix}
\begin{bmatrix}
w_1\\w_2
\end{bmatrix}
= \begin{bmatrix}
c+G\rvv{(\alpha_1)}x(\alpha_0)\\Ax(\alpha_0)
\end{bmatrix}
\end{split}\label{eq:lin_sys_example}
\end{align}
where $G(\alpha_1) = \operatorname*{diag}(\alpha_1,0,0)$, $A = \operatorname*{col}(1,1,1)^{\top}$, $c = \operatorname*{col}(0,1,1)$, and choosing $\nu(\alpha_1)=-w_2$, $x(\alpha_1)=x(\alpha_0)-w_1$. This yields $w_2=1$, $w_{1,1} = 1/\alpha_0-1/\alpha_1$, $w_{1,2}+w_{1,3}=-w_{1,1}$. Choosing the least squares solution of \cref{eq:lin_sys_example}, we get $w_{1,2}=w_{1,3}=0.5/\alpha_1-0.5/\alpha_0$, meaning that 
\begin{align*}
\nu(\alpha_1) = \nu(\alpha_0) = -1,~~ x(\alpha_1) = x(\alpha_0) + d (1/\alpha_1-1/\alpha_0),
\end{align*}
with $d=\operatorname*{col}(1,-1/2,-1/2)$. We therefore have
\begin{align*}
\lim_{\alpha_1 \to \alpha_0} \frac{x(\alpha_1)-x(\alpha_0)}{\alpha_1-\alpha_0} =
d \lim_{\alpha_1 \to \alpha_0}\frac{1/\alpha_1-1/\alpha_0}{\alpha_1-\alpha_0} = \rvv{-} d / \alpha_0^2.
\end{align*}
This indicates that $x(\cdot)$ is differentiable at $\alpha_0$ if the solution $x(\alpha_1)$ is computed by solving \cref{eq:lin_sys_example} in the least squares sense. Geometrically, this corresponds to selecting, for the new value of $\alpha$, the point that remains closest to the reference solution $x(\alpha_0)$ while still satisfying the KKT conditions. As a result, even if the solution map of \cref{eq:QP_example} is set-valued and not differentiable in the classical sense, it is still possible to extract meaningful derivatives associated with the particular solution returned by the solver in \cref{eq:lin_sys_example} for a given $x(\alpha_0)$. This viewpoint will be further developed in \cref{example:2}.\demo
\end{example}

When SSOSC, LICQ, or SCS are not satisfied, \cref{eq:derivative} generally admits infinitely many solutions.
Although one can select the least squares solution, as illustrated in \cref{section:sim}, \cref{example:3}, this may produce a ``derivative'' inconsistent with finite-difference approximations.
The discrepancy is due to the fact that the least squares solution is not tied to any particular local minimizer under perturbations of the parameters.
Instead, our approach is based on a modified system of linear equations, which selects, among the infinitely many possible solutions of \cref{eq:derivative}, one that is consistent with the \emph{nominal} solution of the NLP, yielding a practically meaningful first-order characterization of the perturbed minimizer.
%
\section{Regularization of the NLP}\label{section:reg}
In this section we construct a differentiable-by-design regularized nonlinear program whose solution coincides with the one of \probone at $\theta$. 
\rvv{First, we turn \probone into an equivalent box-constrained program
\begin{align}
\begin{split}
\operatorname*{minimize}_{x,z} & \quad f_\theta(x)\\
\text{subject to} & \quad c_\theta(x,z) = 0, \\
& \quad \rvv{z \leq 0}.
\end{split}\label{eq:P2} \tag{$P_2(\theta)$}
\end{align}
where $z\in\R^{n_\text{in}}$ and $c_\theta(x,z) = \operatorname*{col}(g_\theta(x)-z, h_\theta(x))$.
We denote with $\mu\in\R^{n_\mu}$ and $\delta\in\R^{n_\text{in}}$ the dual variables of \probtwo, where $n_\mu=n_\text{in}+n_\text{eq}$.}
This reformulation is well-known in the optimization literature and, given any $\phi=(x,\lambda,\nu)\in \Phi (\theta)$, one can show that the vector \rvv{$(x,z,\mu,\delta)$} with \rvv{$z_i=g_{\theta,i}(x)$}, $\mu=(\lambda,\nu)$\rvv{, and $\delta=\lambda$} is a primal-dual solution of \probtwo. Moreover, if $\phi$ \rvv{fulfills} the SSOSC, LICQ, and SCS for \probone, then \rvv{$(x,z,\mu,\delta)$} fulfills the same conditions for \probtwo. We additionally define 
\begin{align*}
\Xi(\theta) = \{ \xi=(x,z,\mu) : (x,\mu)\in \Phi(\theta), \rvv{z_i = g_{\theta,i}(x)} \}.
\end{align*}
Inspired by \cite{de2023regularized}, we add regularization terms to \probtwo to ensure differentiability by design. For any $\bar{\xi}=(\bar{x},\bar{z},\bar{\mu})\in \Xi(\theta)$, \rvv{with $\bar{\mu}=(\bar{\lambda},\bar{\nu})$, let $\mathcal{I}_+:=\mathcal{I}_+(\bar{x},\bar{\lambda})$, $\mathcal{I}_0 := \mathcal{I}_0(\bar{x},\bar{\lambda})$, and define
\begin{subequations}
\label{eq:degenerate_subspaces}\begin{align}
\mathcal{U} & := \{ u=(u_\text{in},u_\text{eq})\in\R^{n_\mu} : u_{\text{in},i}=0,\ \forall i\notin\mathcal{I}, \notag \\ & \quad\quad \nabla_x g_\theta(\bar{x})^\top u_\text{in} + \nabla_x h_\theta(\bar{x})^\top u_\text{eq} = 0 \}, \label{eq:degenerate_subspaces:U} \\
\mathcal{C} & := \{ p\in\R^{n_x} : \nabla_x g_{\theta,i}(\bar{x})p = 0,\ \forall i\in\mathcal{I}_+,~ \notag \\ & \quad \quad \nabla_x g_{\theta,i}(\bar{x})p \leq 0,\ \forall i\in\mathcal{I}_0,~ \nabla_x h_\theta(\bar{x})p = 0 \},\\
\mathcal{P} & := \operatorname*{span}\{ p\in\mathcal{C} : p^\top \nabla_{xx}^2 \bar{\mathcal{L}}_\theta(\bar{\phi}) p = 0 \},
\end{align}
\end{subequations}
and let $\Pi_\mathcal{U}$ and $\Pi_\mathcal{P}$ be the orthogonal projectors onto the subspaces $\mathcal{U}$ and $\mathcal{P}$. 
The set $\mathcal{C}$ is the critical cone of \probone at $\bar{\phi}$, and $\mathcal{P}$ is spanned by the critical directions along which the Hessian of the Lagrangian is flat.
The set $\mathcal{U}$ is trivial if and only if the LICQ holds at $\bar{x}$, and under \cref{ass:SONC} below, $\mathcal{P}$ is trivial if and only if the SOSC holds at $\bar{\phi}=(\bar x,\bar \lambda,\bar \nu)$.
In particular, if \probone is nondegenerate, then $\Pi_\mathcal{P}=0$ and $\Pi_\mathcal{U}=0$.} 
\rvv{Consider the problem}
\begin{align}
\begin{split}
\operatorname*{minimize}_{\xi=(x,z,\mu)} & \quad f_\theta(x) + \rvv{\frac{\rho_1}{2} \|x-\bar{x}\|_{\Pi_\mathcal{P}}^2+\frac{\rho_2}{2}\|\mu\|^2}\\
\text{subject to} & \quad c_\theta(x,z) + \rvv{\rho_2\Pi_\mathcal{U}}(\bar{\mu}-\mu) = 0,\\
& \quad \rvv{z \leq 0},
\end{split}\tag{$P_3(\theta,\bar{\xi},\rho)$}\label{eq:P3}
\end{align}
with \rvv{$\rho=(\rho_1,\rho_2)$, $\rho_1,\rho_2>0$}.
The regularized problem differs from \probtwo because of the proximal terms penalizing deviations from $\bar{\xi}$ \rvv{along the degenerate subspaces}.
If the value of $\theta$ is slightly perturbed, the regularization ensures that, whenever \probtwo admits multiple solutions in the vicinity of $\bar{\xi}$, \longprobthree will choose \rvv{one that is close to} $\bar{\xi}$. 
Therefore, problem \probthree can be used to compute a surrogate derivative for the solution map \probone even when the latter is not differentiable. 
Crucially, we need not solve \probthree but \emph{only differentiate its solution map}.

The Lagrangian of \longprobthree is given by
\begin{align}
\mathcal{L}_{\theta}^\rho(\rvv{\eta};\bar{\xi}) = &~ f_\theta(x) + \rvv{\frac{\rho_1}{2} \|x-\bar{x}\|_{\Pi_\mathcal{P}}^2+\frac{\rho_2}{2}\|\mu\|^2} \notag \\
& ~ + \sum_{i=1}^{n_\mu} \psi_i c_{\theta,i}(x,z) + \rvv{\rho_2\psi^\top\Pi_\mathcal{U}(\bar{\mu}-\mu)} + \rvv{\delta^\top z}, \label{eq:P3_lagrangian}
\end{align}
where $\psi=(\psi_\text{in},\psi_\text{eq})\in\R^{n_\mu}$ \rvv{and $\delta\in\R^{n_\text{in}}$ are the multipliers of \probthree}. Denoting the primal-dual variable with $\eta=(\xi,\psi\rvv{,\delta})$, the KKT conditions of \longprobthree are
\begin{align}
\begin{split}
\nabla_x f_\theta(x)+\rvv{\rho_1\Pi_\mathcal{P}}(x-\bar{x})+\sum_{i=1}^{n_\mu}\psi_i \nabla_x c_{\theta,i}(x,0) & = 0 ,\\
\rvv{- \psi_\text{in} + \delta} & = 0 ,\\
\rvv{\rho_2(\mu-\Pi_\mathcal{U}\psi)} & = 0 ,\\
c_\theta(x,z) + \rvv{\rho_2\Pi_\mathcal{U}}(\bar{\mu}-\mu) & = 0,\\
\rvv{\forall i,~ z_i \leq 0,~~ \delta_i \geq 0,~~ z_i \delta_i} & = 0.
\end{split}\label{eq:KKT_3}
\end{align}
\rvv{Where we used $\nabla_z c_\theta(x,z)=\operatorname*{col}(-I,0)$. For our regularization scheme we require the following assumption.}
\rvv{\begin{assumption}\label{ass:SONC}
The primal-dual pair $\bar{\phi}\in \Phi(\theta)$ satisfies the constant rank constraint qualification (CRCQ), that is, for any subset $M \subseteq \Z_{[1,n_\text{eq}]}$ and $N \subseteq \mathcal{I}(\bar{x})$, the family of gradients $\{\nabla g_{\theta,i}(x)\}_{i\in N} \cup \{\nabla h_{\theta,i}(x)\}_{i\in M}$ remains of constant rank for $x$ near $\bar{x}$.
\end{assumption}}
\rvv{\cref{ass:SONC} in particular requires that \probone admits a feasible point. This may not be the case if the NLP is embedded in an iterative policy optimization algorithm, where updates to the parameter $\theta$ may lead to a loss of feasibility. In this paper, we do not consider such scenarios, which are left as future work.}
\rvv{By \cite[Theorem 3.1]{Andreani2010}, \cref{ass:SONC} implies that the second-order necessary conditions hold at $\bar{\phi}$ (in fact, at every multiplier associated with $\bar{x}$), that is,
\begin{align}
p^\top \nabla_{xx}^2 \bar{\mathcal{L}}_\theta(\bar{\phi}) p\geq 0 \quad \text{for all } p\in\mathcal{C}. \label{eq:SONC_P1}
\end{align}
Note that the set of critical directions along which this form vanishes need not be a subspace when $\mathcal{I}_0(\bar{x},\bar{\lambda})\neq\emptyset$, which is why $\mathcal{P}$ is defined as its span.}
\rvv{\cref{ass:SONC} is strictly weaker than the LICQ, and it is generally considered relatively mild. In particular, it is significantly weaker than the assumptions required in \cref{lemma:differentiability}, or even the milder ones in \cite{LiaoMcPherson2020}.}

\rvv{Both projectors can be obtained from the nominal solution: $\Pi_\mathcal{U}$ from the singular value decomposition of the active constraint gradients, and $\Pi_\mathcal{P}$ from the eigendecomposition of the Hessian of the Lagrangian, restricted to the critical cone. 
If the active set is miscalculated by including more elements than needed in $\mathcal{U}$ or $\mathcal{P}$, then the derivative will be biased in directions that are non-degenerate, with a bias that vanishes as $\rho \to 0$.
Since this is preferrable over the opposite case of neglecting active constraints, tight tolerances should be used to identify the active set.}

\rvv{\begin{proposition}\label{lemma:regularized_problem_is_diff_ae}
Let $\bar{\phi}=(\bar{x},\bar{\lambda},\bar{\nu})\in \Phi(\theta)$ satisfy \cref{ass:SONC}. 
Let $\eta=(\xi^*,\bar{\mu},\bar{\lambda})$, where $\xi^*=(\bar{x},g_\theta(\bar{x}),\Pi_\mathcal{U}\bar{\mu})$ and $\bar{\mu}=(\bar{\lambda},\bar{\nu})$.
Then, for any $\rho_1,\rho_2>0$, $\eta$ is a KKT point of \longprobthree satisfying the LICQ. 
Moreover, for any $\rho_1>0$ there exists some $\bar{\rho}_2>0$ such that for all $\rho_2\in(0,\bar{\rho}_2]$, $\eta$ satisfies the SOSC, and therefore $\xi^*$ is a strict local minimizer of \longprobthree.
If $\bar{\phi}$ additionally satisfies the SCS, then $\eta$ satisfies the SCS and the SSOSC for all $\rho_1,\rho_2>0$.
\end{proposition}
\begin{proof} Please see Appendix\,\ref{section:proof_prop_1}.\end{proof}}
Observe that the effect of the regularization in \longprobthree does not impact the solution itself, which matches the one of \probone provided $\bar{\xi}\in \Xi(\theta)$; however, the derivatives of the two solution maps may differ. This is due to the proximal terms in the cost, which force the solution map to stay close to $\bar{\xi}$ for local variations of $\theta$.

\rvv{By strengthening \cref{ass:SONC}, we can ensure that the regularized problem additionally satisfies the SSOSC. Let
\begin{align*}
\mathcal{K}_+ & := \{ p\in\R^{n_x} : \nabla_x g_{\theta,i}(\bar{x})p = 0,\ \forall i\in\mathcal{I}_+, \nabla_x h_\theta(\bar{x})p = 0 \}. \end{align*}
\begin{assumption}\label{ass:strong_SONC}
The primal-dual pair $\bar{\phi}\in \Phi(\theta)$ satisfies \cref{ass:SONC} and $p^\top \nabla_{xx}^2 \bar{\mathcal{L}}_\theta(\bar{\phi}) p\geq 0$ for all $p\in\mathcal{K}_+$.
\end{assumption}
The subspace $\mathcal{K}_+$ is obtained from the critical cone $\mathcal{C}$ by dropping the weakly active constraints, so that $\mathcal{C}\subseteq\mathcal{K}_+$. 
Under SCS the two sets (and the Assumptions) coincide.
Under \cref{ass:strong_SONC}, the set $\mathcal{P}_+ := \{ p\in\mathcal{K}_+ : p^\top \nabla_{xx}^2 \bar{\mathcal{L}}_\theta(\bar{\phi}) p = 0 \}$ is a subspace, and $\mathcal{P}\subseteq\mathcal{P}_+$, with equality under SCS.}
\rvv{\begin{corollary}\label{cor:SSOSC}
Let $\bar{\phi}\in \Phi(\theta)$ satisfy \cref{ass:strong_SONC}, and let \longprobthree be defined with $\Pi_{\mathcal{P}_+}$ in place of $\Pi_\mathcal{P}$. Then the conclusions of \cref{lemma:regularized_problem_is_diff_ae} hold and, moreover, for any $\rho_1>0$ there exists $\bar{\rho}_2>0$ such that, for all $\rho_2\in(0,\bar{\rho}_2]$, $\eta$ satisfies the SSOSC for \longprobthree.
\end{corollary}
\begin{proof}
The SSOSC requires the quadratic form \cref{eq:SOSC_form}, with $\Pi_{\mathcal{P}_+}$ in place of $\Pi_\mathcal{P}$, to be positive for all $v\neq 0$ satisfying $v_{z,i}=0$, for all $i\in\mathcal{I}_+$, and $B v_{xz} = \rho_2\Pi_\mathcal{U}v_\mu$, that is, \cref{eq:SOSC_constraint} without the inequality condition on weakly active constraints.
As in the proof of \cref{lemma:regularized_problem_is_diff_ae}, this reduces to the positivity of $v_{xz}^\top Q(\rho_2)v_{xz}$ on the subspace $\{v_{xz} : v_{z,i}=0,\ \forall i\in\mathcal{I}_+\}$.
If $Bv_{xz}=0$ and $v_{z,i}=0$ for $i\in\mathcal{I}_+$, then $v_z=\nabla_x g_\theta(\bar{x})v_x$ and $v_x\in\mathcal{K}_+$, so by \cref{ass:strong_SONC} and the definition of $\mathcal{P}_+$ the first term of $Q(\rho_2)$ is positive unless $v_{xz}=0$.
Finsler's lemma \cite[Lemma 4.28]{ruszczynski2011nonlinear} then yields $\bar{\rho}_2>0$ with the desired property.
\end{proof}}
%
\section{Differentiating the regularized NLP}
\cref{lemma:regularized_problem_is_diff_ae} proves that \longprobthree satisfies the LICQ and \rvv{SOSC conditions under a mild condition (\cref{ass:SONC}) and a judicious choice of $\rho$.}
\rvv{SCS and SSOSC additionally hold if the SCS holds at $\bar{\phi}$.}
\rvv{In the latter case, the solution map of \probthree is differentiable by \cref{lemma:differentiability}.}
In this section we provide a simple method to compute the derivative of \probthree with respect to $\theta$, and prove that, if \rvv{\probone is nondegenerate}, then the two derivatives \rvv{coincide}.
Moreover, we show that obtaining the derivative of \probthree amounts to solving a linear system of equations that is similar in structure to the one in \cref{lemma:differentiability}.

To compute the derivative of \probthree, we apply the implicit function theorem to the KKT conditions in \cref{eq:KKT_3}.
\begin{proposition}\label{prop:regularized_problem_is_cont_diff}
\rvv{Suppose $\bar{\phi}\in \Phi(\theta)$ satisfies \cref{ass:SONC} and SCS. 
Then the} primal-dual solution map $\eta(\theta)$ of \longprobthree is continuously differentiable at $\theta$ for all \rvv{$\rho_1,\rho_2>0$} \rvv{and the} derivative of $x(\theta)$ with respect to $\theta$ is given by the first $n_x$ rows of \rvv{the} unique matrix $V$ solving the \rvv{symmetric} linear system of equations \rvv{$\tilde{A}_\theta^\rho(\bar{\phi})V=b(\bar{\phi})$, where
\begin{align}\label{eq:approx_derivative}
\tilde{A}_\theta^\rho(\bar{\xi}) & = \begin{bmatrix}
    \nabla_{xx}^2 \bar{\mathcal{L}}_\theta(\bar{\phi})\!+\!\rho_1 \Pi_\mathcal{P} & \nabla_x g_{\theta,\mathcal{I}}(\bar{x})^{\top} & \nabla_x h_\theta(\bar{x})^{\top} \\
    \nabla_x g_{\theta,\mathcal{I}}(\bar{x}) & -\rho_2\Pi_\mathcal{U}^{\mathcal{I}\mathcal{I}} & -\rho_2\Pi_\mathcal{U}^{\mathcal{I}\mathrm{eq}} \\
    \nabla_x h_\theta(\bar{x}) & -\rho_2\Pi_\mathcal{U}^{\mathrm{eq}\mathcal{I}} & -\rho_2 \Pi_\mathcal{U}^{\mathrm{eq}\,\mathrm{eq}}
\end{bmatrix},
\end{align}
where $\Pi_\mathcal{U}^{\mathcal{I}\mathcal{I}}$, $\Pi_\mathcal{U}^{\mathcal{I}\mathrm{eq}}$, $\Pi_\mathcal{U}^{\mathrm{eq}\mathcal{I}}$, and $\Pi_\mathcal{U}^{\mathrm{eq}\,\mathrm{eq}}$ are the submatrices of $\Pi_\mathcal{U}$ whose rows and columns are indexed by the active inequality components and the equality components, respectively.}
\end{proposition}
\begin{proof} Please see Appendix\,\ref{section:proof_prop_2}. \end{proof}
\rvv{When SCS is lost, we can still obtain directional derivatives in all directions under \cref{ass:strong_SONC}.}
\rvv{\begin{proposition}\label{prop:directional_derivative}
Under \cref{ass:strong_SONC}, for any $\rho_1>0$ there exists $\bar{\rho}_2>0$ such that, for all $\rho_2\in(0,\bar{\rho}_2]$, $\eta=(\xi^*,\bar{\mu},\bar{\lambda})$ satisfies the LICQ and the SSOSC for \longprobthree, and the primal-dual solution map $\eta(\theta)$ of \probthree (with $\Pi_{\mathcal{P}_+}$ in place of $\Pi_\mathcal{P}$) is directionally differentiable at $\theta$. Moreover, for any $d\in\R^{n_\theta}$, the directional derivatives of $x(\theta)$ and $\Pi_\mathcal{U}\mu(\theta)$ in the direction $d$ are given by the unique solution $(v_x,w)$ of the quadratic program
\begin{align}\label{eq:directional_derivative_regularized}
\begin{split}
\operatorname*{min.}_{v_x,w} & ~~ \frac{1}{2}v_x^\top\big(\nabla_{xx}^2 \bar{\mathcal{L}}_\theta(\bar{\phi})+\rho_1\Pi_{\mathcal{P}_+}\big)v_x \! + v_x^\top \nabla^2_{\theta x}\bar{\mathcal{L}}_\theta(\bar{\phi})d + \frac{\rho_2}{2}\|w\|^2\\
\text{\textup{s.t.}} & ~~ \nabla_x g_{\theta,i}(\bar{x})v_x + \nabla_\theta g_{\theta,i}(\bar{x})d = \rho_2 w_{\text{in},i}, \quad i\in\mathcal{I}_+,\\
& ~~ \nabla_x g_{\theta,i}(\bar{x})v_x + \nabla_\theta g_{\theta,i}(\bar{x})d \leq \rho_2 w_{\text{in},i}, \quad i\in\mathcal{I}_0,\\
& ~~ \nabla_x h_\theta(\bar{x})v_x + \nabla_\theta h_\theta(\bar{x})d = \rho_2 w_\text{eq},\\
& ~~ w\in\mathcal{U}.
\end{split}
\end{align}
\end{proposition}
\begin{proof} Please see Appendix\,\ref{section:proof_prop_3}. \end{proof}}

The linear system in \cref{eq:approx_derivative} is similar to the one in \cref{lemma:differentiability} with an additional symmetric perturbation of the left-hand side \rvv{acting along the degenerate subspaces}.
This perturbation disappears \rvv{when \probone is nondegenerate}. 
\rvv{If \probone is solved using a sequential quadratic programming method with an active-set QP solver and the Hessian is computed exactly} (as described e.g. in \cite[Section 16.5]{nocedal1999numerical}), then a factorization of $A_\theta(\bar{\phi})$ \rvv{may be} already available for use. 
This could be exploited to speed up the resolution of \cref{eq:approx_derivative}, especially since \rvv{the perturbation is a low-rank symmetric matrix}.
This is beyond the scope of this paper and will be treated in future work.

\rvv{If \probone is non-degenerate, the surrogate derivative obtained through \cref{eq:approx_derivative} coincides with the true derivative of the solution map, independently of the choice of $\rho_1$ and $\rho_2$.
In the following we denote with $\bar{\nabla}_\theta^\rho x(\theta;\bar{\xi})$ and $\bar x^{\rho '}(\theta,\bar{\xi};d)$ the surrogate derivative and directional derivative of the solution map of the regularized problem, respectively, and with $\nabla_\theta x(\theta)$ and $x'(\theta;d)$ the derivative and the directional derivative of the true solution map (assuming they exist).}
\begin{theorem}\label{theorem:bound}
Suppose $\bar\phi\in \Phi(\theta)$ satisfies the SSOSC \rvv{and} LICQ for \probone. Then, for all \rvv{$\rho_1,\rho_2>0$ and all $d\in\R^{n_\theta}$, the directional derivative $\bar{x}^{\rho\,\prime}(\theta;d)$ of the $x$-component of the solution map of \longprobthree (with $\Pi_{\mathcal{P}_+}$ in place of $\Pi_\mathcal{P}$) coincides with the directional derivative $x'(\theta;d)$ of the solution map of \probone. If, additionally, $\phi$ satisfies the SCS, then $\bar{\nabla}_\theta^\rho x(\theta;\bar{\xi}) = \nabla_\theta x(\theta)$.}
\end{theorem}
\begin{proof} Please see Appendix\,\ref{section:proof_thm_1}. \end{proof}
%
\section{Experiments}\label{section:sim}
We start by showing that our algorithm deployed on \cref{example:1} successfully approximates the derivative of a specific solution of the QP in \cref{eq:QP_example} obtained through \cref{eq:lin_sys_example}.
\rvv{The code used to produce these results is available open-source at the link \url{https://github.com/RiccardoZuliani98/dbdnlp}.}

\begin{example}[{Continuation of \cref{example:1}}]\label{example:2}
\rvv{Problem \cref{eq:QP_example} has no inequality constraints, and the gradient $A=(1,1,1)$ of its equality constraint is nonzero. Hence the LICQ holds and $\mathcal{U}=\{0\}$. The Hessian of the Lagrangian is $G(\alpha)=\operatorname*{diag}(\alpha,0,0)$, and the critical cone is $\mathcal{C}=\{p\in\R^3: p_1+p_2+p_3=0\}$. Since $p^\top G(\alpha)p=\alpha p_1^2$ vanishes on $\mathcal{C}$ if and only if $p_1=0$, we obtain
\begin{align*}
\mathcal{P}=\operatorname{span}\{(0,1,-1)\}, ~~
\Pi_{\mathcal{P}}=\frac{1}{2}\begin{bmatrix}
0&0&0\\0&1&-1\\0&-1&1
\end{bmatrix}, ~~ \Pi_{\mathcal{U}}=0.
\end{align*}
The degeneracy of \cref{eq:QP_example} is therefore entirely due to the flat direction along which $x_2$ and $x_3$ can be traded off, which makes the SSOSC fail. Moreover, $\nabla^2_{\alpha x}\bar{\mathcal{L}}_\alpha(\phi(\alpha))=(x_1(\alpha),0,0)$ and the constraint does not depend on $\alpha$. Therefore, for a given $\rho_1>0$, the system in \cref{eq:approx_derivative} is given by
\begin{align*}
\begin{bmatrix}
\alpha&0&0&1\\
0&\rho_1/2&-\rho_1/2&1\\
0&-\rho_1/2&\rho_1/2&1\\
1&1&1&0
\end{bmatrix} v = 
\begin{bmatrix}
-x_1(\alpha)\\0\\0\\0
\end{bmatrix} =
\begin{bmatrix}
-1/\alpha\\0\\0\\0
\end{bmatrix},
\end{align*}
where $\rho_2$ does not appear since $\Pi_{\mathcal{U}}=0$. Adding and subtracting the second and third rows gives $v_4=0$ and $v_2=v_3$. The first and last rows then yield
\begin{align*}
v_1 = -\frac{1}{\alpha^2}, ~~ v_2 = v_3 = \frac{1}{2\alpha^2}, ~~ v_4 = 0,
\end{align*}
for all $\rho_1>0$. The first three components coincide with the derivative $-d/\alpha^2$ obtained in \cref{example:1}. The last one is consistent with $\nu(\alpha)=-1$ being constant. Hence, our method returns exactly the derivative of the specific solution of \cref{eq:QP_example} obtained by solving \cref{eq:lin_sys_example} in the least squares sense, independently of $\rho$. 

Observe that setting $\rho_1=0$ does not suffice. The matrix above then becomes singular, with null space spanned by $(0,1,-1,0)$, and \cref{eq:derivative} admits infinitely many solutions satisfying $v_1=-1/\alpha^2$, $v_4=0$, and $v_2+v_3=1/\alpha^2$. The regularization selects, among these, the one with no component along $\mathcal{P}$.}\demo
\end{example}

\begin{example}[Trajectory optimization]\label{example:3}
Next, we consider \rvv{an} optimal control example that does not fulfill the conditions for differentiability in \cref{lemma:differentiability}.
Consider a nonlinear continuous-time model of a car adapted from \cite{thierry2020l1}
\begin{align}
\begin{split}
\ddot{p}_x(t) & = u_1(t) \cos\vartheta(t),\\
\ddot{p}_y(t) & = u_1(t) \sin\vartheta(t),\\
\dot{\vartheta}(t) & = u_2(t)(\dot{p}_x(t)\cos\vartheta(t) + \dot{p}_y(t)\sin\vartheta(t)),
\end{split}\label{eq:car_ode}
\end{align}
where $p_x(t)$ and $p_y(t)$ are the $x$ and $y$ position of the car, $\vartheta(t)$ is the orientation with respect to the positive $x$-axis, $u_1(t)$ and $u_2(t)$ are the acceleration and steering control, respectively, where $u(t)=\operatorname*{col}(u_1(t),u_2(t))$ is the control input.

We consider the problem of steering the car from an initial resting position $p_x(0)=p_y(0)=\vartheta(0)=0$, to a terminal position $p_x(T)=0.5$, $p_y(T)=0.25$ with $\dot{p}_x(T)=\dot{p}_y(T)=0$ while minimizing the time $T$. The input must satisfy the constraints \rvv{$|u_1(t)| \leq 1$, $|u_2(t)|\leq \theta$}, where $\theta$ is a parameter\rvv{. Since the steering bound is active at the optimal solution, perturbing $\theta$ moves an active constraint}.

Let $x=\operatorname*{col}(p_x,p_y,\dot{p}_x,\dot{p}_y,\vartheta)$ be the state of the system, and let the dynamics be given by $\dot{x}(t)=f(x(t),u(t))$, where $f$ is derived from \cref{eq:car_ode}. To express this free-final-time problem in the form in \probone, we express the ODE in normalized time $\tau\in[0,1]$, and define the \emph{time dilation} $\sigma=(d \tau / dt)^{-1}$, such that $\dot{x}(\tau):=\frac{d}{d \tau} x(\tau) = \sigma f(x(\tau),u(\tau))$ \cite{szmuk2018successive}. In this way, given the number of time-steps $N$, we can discretize the ODE $\dot{x}(\tau)=\sigma f(x(\tau),u(\tau))$ using \rvv{the explicit Euler method} with a sampling time $1/(N+1)$ to obtain \rvv{$T=N\sigma/(N+1)$}. As a result, we can formulate the optimal control problem as follows:
\begin{align}
\begin{split}
\operatorname*{min.}_{x,u,\sigma} & \quad \sigma\\
\text{s.t.} & \quad x_{t+1} = f_{\text{d}}(x_t,u_t), ~~t\in\Z_{[0,\rvv{N-1}]},\\
& \quad \rvv{|u_{1,t}| \leq 1,~~ |u_{2,t}| \leq \theta,~~t\in\Z_{[0,N-1]}},\\
& \quad \rvv{\sigma \geq 0},\\
& \quad p_{x,0}=p_{y,0}=0,~~\dot{p}_{x,0}=\dot{p}_{y,0}=0,\rvv{~~\vartheta_0=0},\\
& \quad p_{x,\rvv{N}}=0.5,~~p_{y,\rvv{N}}=0.25,\\
& \quad \dot{p}_{x,\rvv{N}}=\dot{p}_{y,\rvv{N}}=0,
\end{split}\label{eq:P_car}
\end{align}
where $f_{\text{d}}$ is obtained by discretizing $\sigma f$, and $x_t=(p_{x,t},p_{y,t},\dot{p}_{x,t},\dot{p}_{y,t},\vartheta_t)$ and $u_t=(u_{1,t},u_{2,t})$ denote the discrete-time state and input at each time step, respectively.

\rvv{The explicit Euler discretization makes \cref{eq:P_car} structurally degenerate. Since the car starts at rest, the first steering input $u_{2,0}$ has no effect on the dynamics (it multiplies the velocity $(\dot{p}_{x,0},\dot{p}_{y,0})=0$), and the last steering input $u_{2,N-1}$ only affects the final heading $\vartheta_N$, which is free. Both are flat directions of the Lagrangian on the critical cone, so that $\dim\mathcal{P}=2$, the SSOSC fails, and the unregularized KKT matrix is singular.}

\cref{fig:car_simulations} depicts the optimal solution of \cref{eq:P_car} with $\bar{\theta}=1$ and $N=150$ obtained by solving \cref{eq:P_car} in IPOPT.
The Figure additionally shows \rvv{the local minimizer for $\theta=\bar{\theta}+\delta \theta$ with $\delta \theta = 0.1$, obtained by warm-starting the solver with the nominal solution obtained with $\theta=\bar{\theta}$.}
The first-order approximation $x(\bar{\theta})+\bar{\nabla}_\theta^\rho x(\bar{\theta})\delta \theta$\rvv{, with $\rho_1=\rho_2=10^{-4}$}, shown in red, almost perfectly matches the minimizer $x(\bar{\theta}+\delta \theta)$ obtained by \rvv{warm-starting} the solver with the nominal solution, even though $|\delta \theta|$ is not negligible compared to $\bar{\theta}$ and the solution map is not differentiable due to the failure of the SSOSC \rvv{(recall that $\dim\mathcal{P}=2$)}.
\rvv{Observe that the first-order approximation is unaware of the input constraints, and therefore exhibits a small overshoot upon input saturation (see inset in the bottom plot of \cref{fig:car_simulations}).}

\begin{figure}
\centering
\begin{tikzpicture}

\def\dubinsTf{2.5}

\pgfplotsset{m/.style={line5}}
\pgfplotsset{m2/.style={line4}}
\pgfplotsset{app/.style={line6}}

\pgfplotsset{every axis plot/.append style={mark repeat=10}}

\pgfmathsetmacro{\zoomTopXmin}{0.50*\dubinsTf}
\pgfmathsetmacro{\zoomTopXmax}{0.65*\dubinsTf}
\def\zoomTopYmin{-0.10}
\def\zoomTopYmax{0.10}
\pgfmathsetmacro{\zoomBotXmin}{0.55*\dubinsTf}
\pgfmathsetmacro{\zoomBotXmax}{0.65*\dubinsTf}
\def\zoomBotYmin{-1.75}
\def\zoomBotYmax{-0.70}

\pgfplotsset{inset/.style={
    scale only axis,
    width=2.4cm,
    height=0.95cm,
    axis background/.style={fill=white},
    axis line style={black!70,solid,semithick},
    ticks=none,
}}
\tikzset{zoombox/.style={black!70, semithick, dashed}}

\begin{groupplot}[group style={group size=1 by 4, vertical sep = 10pt, group name=plots}, ylabel style={at={(-0.12,0.5)}}]
\nextgroupplot[
clip mode=individual,
height=3.25cm,
label style={font=\footnotesize},
legend cell align={left},
legend style={
  fill opacity=0.8,
  draw opacity=1,
  text opacity=1,
  at={(0.03,0.97)},
  anchor=north west,
  draw=lightgray204,
  row sep = -2.5pt
},
legend style={font=\footnotesize},
scaled x ticks=manual:{}{\pgfmathparse{#1}},
tick align=outside,
tick label style={font=\footnotesize},
width=8cm,
x grid style={darkgray176},
xmajorgrids,
xmajorticks=false,
xmin=0, xmax=\dubinsTf,
xtick style={color=black},
xticklabels={},
y grid style={darkgray176},
ylabel={$x$-position [m]},
ymajorgrids,
ymin=-0.21, ymax=0.55,
ytick pos=left,
ytick style={color=black},
]
\addplot [m2]
table {Figures/Data/x_pos_m2.tex};
\addplot [m]
table {Figures/Data/x_pos.tex};
\addplot [app]
table {Figures/Data/x_pos_m2_approx.tex};
\draw [zoombox,solid] (axis cs:\zoomTopXmin,\zoomTopYmin) rectangle (axis cs:\zoomTopXmax,\zoomTopYmax);
\coordinate (topZoomNE) at (axis cs:\zoomTopXmax,\zoomTopYmax);
\coordinate (topZoomSW) at (axis cs:\zoomTopXmin,\zoomTopYmin);

\nextgroupplot[
clip mode=individual,
height=3.25cm,
label style={font=\footnotesize},
legend cell align={left},
legend style={
  fill opacity=0.8,
  draw opacity=1,
  text opacity=1,
  at={(0.03,0.97)},
  anchor=north west,
  draw=lightgray204,
  row sep = -4pt
},
legend style={font=\scriptsize},
scaled x ticks=manual:{}{\pgfmathparse{#1}},
tick align=outside,
tick label style={font=\footnotesize},
width=8cm,
x grid style={darkgray176},
xmajorgrids,
xmajorticks=false,
xmin=0, xmax=\dubinsTf,
xtick style={color=black},
xticklabels={},
y grid style={darkgray176},
ylabel={$y$-position [m]},
ymajorgrids,
ymin=-0.02, ymax=0.265,
ytick pos=left,
ytick style={color=black}
]
\addplot [m2]
table {Figures/Data/y_pos_m2.tex};
\addlegendentry{$x(\bar{\theta}\!+\!\delta \theta)$}
\addplot [m]
table {Figures/Data/y_pos.tex};
\addlegendentry{$x(\bar{\theta})$}
\addplot [app]
table {Figures/Data/y_pos_m2_approx.tex};
\addlegendentry{$x(\bar{\theta})+\bar{\nabla}_\theta^\rho x(\bar{\theta})\delta \theta$}

\nextgroupplot[
clip mode=individual,
height=3.25cm,
label style={font=\footnotesize},
legend cell align={left},
legend style={
  fill opacity=0.8,
  draw opacity=1,
  text opacity=1,
  at={(0.03,0.97)},
  anchor=north west,
  draw=lightgray204,
  row sep = -3.5pt
},
legend style={font=\footnotesize},
scaled x ticks=manual:{}{\pgfmathparse{#1}},
tick align=outside,
tick label style={font=\footnotesize},
width=8cm,
x grid style={darkgray176},
xmajorgrids,
xmajorticks=false,
xmin=0, xmax=\dubinsTf,
xtick style={color=black},
xticklabels={},
y grid style={darkgray176},
ylabel={Throttle $u_1(t)$},
ymajorgrids,
ymin=-1.2, ymax=1.1,
ytick pos=left,
ytick style={color=black}
]
\addplot [m2]
table {Figures/Data/a_m2.tex};
\addplot [m]
table {Figures/Data/a.tex};
\addplot [app]
table {Figures/Data/a_m2_approx.tex};

\nextgroupplot[
clip mode=individual,
legend cell align={left},
legend style={
  fill opacity=0.8,
  draw opacity=1,
  text opacity=1,
  at={(0.03,0.97)},
  anchor=north west,
  draw=lightgray204,
  row sep = -2.5pt
},
height=3.25cm,
label style={font=\footnotesize},
legend style={font=\footnotesize},
tick align=outside,
tick label style={font=\footnotesize},
tick pos=left,
width=8cm,
x grid style={darkgray176},
xlabel={Time [s]},
xmajorgrids,
xmin=0, xmax=\dubinsTf,
xtick style={color=black},
y grid style={darkgray176},
ylabel={Steering $u_2(t)$},
ymajorgrids,
ymin=-1.8, ymax=1.2,
ytick style={color=black}
]
\addplot [m2]
table {Figures/Data/s_m2.tex};
\addplot [m]
table {Figures/Data/s.tex};
\addplot [app]
table {Figures/Data/s_m2_approx.tex};
\draw [zoombox,solid] (axis cs:\zoomBotXmin,\zoomBotYmin) rectangle (axis cs:\zoomBotXmax,\zoomBotYmax);
\coordinate (botZoomNE) at (axis cs:\zoomBotXmax,\zoomBotYmax);
\coordinate (botZoomSW) at (axis cs:\zoomBotXmin,\zoomBotYmin);
\end{groupplot}

\begin{axis}[
  inset,
  name=insetTop,
  at={(plots c1r1.north west)},
  anchor=north west,
  xshift=3pt, yshift=-3pt,
  xmin=\zoomTopXmin, xmax=\zoomTopXmax,
  ymin=\zoomTopYmin, ymax=\zoomTopYmax,
]
\addplot [m2]  table {Figures/Data/x_pos_m2.tex};
\addplot [m]   table {Figures/Data/x_pos.tex};
\addplot [app] table {Figures/Data/x_pos_m2_approx.tex};
\end{axis}

\begin{axis}[
  inset,
  name=insetBot,
  at={(plots c1r4.north west)},
  anchor=north west,
  xshift=6pt, yshift=-3pt,
  xmin=\zoomBotXmin, xmax=\zoomBotXmax,
  ymin=\zoomBotYmin, ymax=\zoomBotYmax,
  width=1.5cm
]
\addplot [m2]  table {Figures/Data/s_m2.tex};
\addplot [m]   table {Figures/Data/s.tex};
\addplot [app] table {Figures/Data/s_m2_approx.tex};
\end{axis}

\draw [zoombox] (topZoomNE) -- (insetTop.north east);
\draw [zoombox] (topZoomSW) -- (insetTop.south west);
\draw [zoombox] (botZoomNE) -- (insetBot.north east);
\draw [zoombox] (botZoomSW) -- (insetBot.south west);

\end{tikzpicture}
\caption{Optimal position and input for $\theta=\bar{\theta}$ (orange), $\theta=\bar{\theta}+\delta \theta$ (blue), and first-order approximation (red).}\label{fig:car_simulations}
\end{figure}

To \rvv{quantify} the effectiveness of our approach, we proceed as follows: 1) we solve \cref{eq:P_car} using IPOPT and obtain a primal-dual solution $(\bar{x},\bar{\lambda},\bar{\nu})$ for $\theta=\bar{\theta}=1$; 2) we solve \cref{eq:approx_derivative} with \rvv{$\rho_1=\rho_2=\rho$ for various values of $\rho$ and obtain an approximate gradient $\bar{\nabla}_{\bar{\theta}}^\rho x(\bar{\theta};\bar{\xi})$ and a first-order prediction $x(\bar\theta)+\bar{\nabla}_{\bar{\theta}}^\rho x(\bar{\theta};\bar{\xi})\delta\theta$; 3) we estimate the true gradient using finite differences, and compute the real solution of \cref{eq:P_car} for $\theta=\bar\theta+\delta\theta$.
Since the minimizers of \cref{eq:P_car} are not isolated, we perform step 3 both \cref{eq:P_car} and on the proximal problem obtained by adding $\frac{\epsilon}{2}\|x-\bar{x}\|^2$, with $\epsilon=10^{-8}$, to the cost of \cref{eq:P_car}, which selects the minimizer closest to $\bar{x}$.
We solve both problems with $\theta^+=\bar{\theta}+\delta \theta$ and $\theta^-=\bar{\theta}-\delta \theta$ (warm-starting with the nominal solution)}, with $\delta \theta = 10^{-5}$, obtaining the solutions $x^+$ and $x^-$, and setting $\nabla_\theta x(\bar{\theta}) \approx \tilde{\nabla}_\theta x(\bar{\theta}):= (x^+-x^-)/2 \delta \theta$.
\rvv{Since the relative error on the derivative is the same for all tested values of $\rho_1=\rho_2=\rho$, we do not report a sweep over $\rho$. Our method achieves a relative error of \dubinsval{\dubinsErrors}{4}{err_prox} with respect to the finite-difference derivative of the proximal variant, and of \dubinsval{\dubinsErrors}{4}{err_warm} with respect to the one of \cref{eq:P_car}. The norm of the true derivative (obtained with finite differences) is $13.72$ in both cases. For comparison, the least squares solution of \cref{eq:derivative} yields a slightly worse relative error of \dubinsval{\dubinsErrors}{7}{err_prox} and \dubinsval{\dubinsErrors}{7}{err_warm} with respect to the proximal variant, and the original problem, respectively. 
The disadvantages of choosing to solve the unregularized system in a least-squares sense are discussed below in \cref{remark:ls}.}
\rvv{We additionally test the accuracy of the first-order approximation over a wider range of perturbations.
\Cref{fig:dubins_taylor} shows the prediction error $\|x(\bar{\theta}+\delta\theta)-x(\bar{\theta})-\bar{\nabla}_\theta^\rho x(\bar{\theta},\bar{\xi})\,\delta\theta\|$ for $\delta\theta\in[10^{-4},10^{-1}]$ measured against the solution $x(\bar{\theta}+\delta\theta)$ of both problems.
Measured against the prox-regularized solution, the error remains small up to $\delta\theta=10^{-2}$ and grows more significantly for $\delta\theta\geq \cdot 10^{-2}$.
Measured against the original problem, instead, the error stagnates around $10^{-2}$ independently of $\delta\theta$: since the solution is not unique, the solver returns a different point of the solution set, and the error does not vanish as $\delta\theta\to 0$. The red curve (right axis) reports the objective gap $f(x_{\mathrm{prox}})-f(x_{\mathrm{nlp}})$, evaluated in the original problem, which remains below $10^{-13}$ over the whole range.
This confirms that, among the optimal solutions, the proximal regularization forces the solver to choose the one closest to the prediction without affecting optimality.}


\begin{figure}
\centering
\begin{tikzpicture}
\pgfplotsset{
  taylor/.style={
    scale only axis, width=0.66\columnwidth, height=3.6cm,
    xmin=7e-5, xmax=1.4e-1,
    tick label style={font=\footnotesize}, label style={font=\small},
  },
}
\begin{loglogaxis}[
  taylor,
  xlabel={$\delta\theta$}, ylabel={Prediction error},
  ylabel style={yshift=-3pt},
  grid=major, axis y line*=left,
  ylabel style={newBlue},
  y axis line style={newBlue}, ytick style={newBlue}, yticklabel style={newBlue},
]
\addplot[line4, mark size=1, very thick]
  table[x=dm, y=projected_vs_prox]{Figures/Data/dubins_taylor.dat};
\label{plot:taylor_prox}
\addplot[line4, dashed, mark size=1, very thick]
  table[x=dm, y=projected_vs_nlp]{Figures/Data/dubins_taylor.dat};
\label{plot:taylor_nlp}
\addplot[black, dash pattern=on 5.25pt off 1.75pt on 1.75pt off 1.75pt, very thick]
  table[x=dm, y expr={\thisrow{dm}^2}]{Figures/Data/dubins_taylor.dat};
\label{plot:taylor_order}
\end{loglogaxis}
\begin{loglogaxis}[
  taylor,
  axis x line=none, axis y line*=right,
  ylabel={Objective gap}, ylabel style={newRed, yshift=3pt},
  y axis line style={newRed}, ytick style={newRed}, yticklabel style={newRed},
  ymax=5e-13,
  legend style={
    at={(0.5,1.03)}, anchor=south, 
    font=\footnotesize, legend columns=2,
    /tikz/every even column/.append style={column sep=0.8em},
  },
  legend cell align=left,
]
\addlegendimage{/pgfplots/refstyle=plot:taylor_prox}
\addlegendentry{vs Prox-Regularized NLP}
\addlegendimage{/pgfplots/refstyle=plot:taylor_nlp}
\addlegendentry{vs Original NLP}
\addlegendimage{/pgfplots/refstyle=plot:taylor_order}
\addlegendentry{$\mathcal{O}(\delta\theta^2)$ left axis}
\addplot[line6, solid, mark size=1.5pt]
  table[x=dm, y expr={abs(\thisrow{projected_gap})}]{Figures/Data/dubins_taylor.dat};
\addlegendentry{Objective gap}
\end{loglogaxis}
\end{tikzpicture}
\caption{\rvv{First-order prediction error (left axis) and objective gap between the prox-regularized and original re-solve (right axis, red) for different perturbations.}}\label{fig:dubins_taylor}
\end{figure}

\rvv{\begin{remark}[An argument against least squares]\label{remark:ls}
One may be tempted to solve \cref{eq:derivative} in the least squares sense. 
This is in fact what \cite{amos2017optnet}, as well as other subsequent methods, suggest.
However, note that the least squares solution depends on how the constraints are scaled. 
To illustrate this, we impose the initial condition as $c\,(p_{x,0},p_{y,0},\dot{p}_{x,0},\dot{p}_{y,0},\vartheta_0)=0$, which describes the same problem for every $c\neq 0$. As shown in \cref{tab:scaling}, the derivative of $u_{2,0}$ obtained with least squares changes with $c$, whereas our method returns the same derivative (that of the minimizer closest to $\bar{x}$) for every $c$.
\end{remark}}

\begin{table}
\centering
\rvv{\begin{tabular}{ c c c c c }
\toprule
& \multicolumn{2}{c}{\textbf{Our method}} & \multicolumn{2}{c}{\textbf{Least squares}} \\
$c$ & $\partial u_{2,0}/\partial\theta$ & \textbf{Rel. error} & $\partial u_{2,0}/\partial\theta$ & \textbf{Rel. error} \\
\midrule
$10^{-2}$ & \dubinsval[sci,precision=2]{\dubinsScaling}{0}{proj_ds0} & \dubinsval{\dubinsScaling}{0}{proj_err_warm} & \dubinsval[sci,precision=2]{\dubinsScaling}{0}{ls_ds0} & \dubinsval{\dubinsScaling}{0}{ls_err_warm} \\
$10^{-1}$ & \dubinsval[sci,precision=2]{\dubinsScaling}{1}{proj_ds0} & \dubinsval{\dubinsScaling}{1}{proj_err_warm} & \dubinsval[sci,precision=2]{\dubinsScaling}{1}{ls_ds0} & \dubinsval{\dubinsScaling}{1}{ls_err_warm} \\
$1$ & \dubinsval[sci,precision=2]{\dubinsScaling}{2}{proj_ds0} & \dubinsval{\dubinsScaling}{2}{proj_err_warm} & \dubinsval[sci,precision=2]{\dubinsScaling}{2}{ls_ds0} & \dubinsval{\dubinsScaling}{2}{ls_err_warm} \\
$10^{1}$ & \dubinsval[sci,precision=2]{\dubinsScaling}{3}{proj_ds0} & \dubinsval{\dubinsScaling}{3}{proj_err_warm} & \dubinsval[sci,precision=2]{\dubinsScaling}{3}{ls_ds0} & \dubinsval{\dubinsScaling}{3}{ls_err_warm} \\
\bottomrule
\end{tabular}}
\caption{\rvv{Derivative of $u_{2,0}$ and relative error when the initial condition is scaled by $c$ ($\rho_1=\rho_2=10^{-4}$).}}\label{tab:scaling}
\end{table}

We further evaluated our method on a higher-dimensional setting by extending $\theta$ to three components, incorporating the desired terminal positions $p_x(T)$ and $p_y(T)$. Using the same value of $\rho$ as before, our approach achieved a relative error of \rvv{\dubinsval{\dubinsThetaThree}{5}{err_warm}.
The least squares solution produces a similar (yet worse) result with a relative error of \dubinsval{\dubinsThetaThree}{6}{err_warm}.}

In terms of efficiency, our method outperforms finite differences, as shown in \cref{tab:comp_times}.
\rvv{When calculating the finite difference approximation, the solver is warmstarted with the nominal solution to improve speed.}
Even if the solution time of each nonlinear optimization problem were reduced using optimization methods that better \rvv{accommodate warm-starting}, the advantage of our approach would persist when the parameter dimension grows. This is because our method only requires solving a set of linear systems with a fixed left-hand side, which scales much more favorably than repeatedly solving nonlinear optimization problems. \demo

\begin{table}
\centering
\begin{tabular}{ c c c }
\toprule
& \textbf{Mean}~\rvv{[s]} & \textbf{Variance}~\rvv{[s$^2$]} \\
\midrule
Our method & \rvv{\dubinsval[sci,precision=2]{\dubinsTimes}{0}{mean} [\dubinsval[sci,precision=2]{\dubinsTimes}{2}{mean}]} & \rvv{\dubinsval[sci,precision=2]{\dubinsTimes}{0}{var} [\dubinsval[sci,precision=2]{\dubinsTimes}{2}{var}]} \\
Finite differences & \rvv{\dubinsval[sci,precision=2]{\dubinsTimes}{1}{mean} [\dubinsval[sci,precision=2]{\dubinsTimes}{3}{mean}]} & \rvv{\dubinsval[sci,precision=2]{\dubinsTimes}{1}{var} [\dubinsval[sci,precision=2]{\dubinsTimes}{3}{var}]} \\
\bottomrule
\end{tabular}
\caption{\rvv{Computation times comparison (100 evaluations) for the one and three dimensional (in brackets) examples.}}\label{tab:comp_times}
\end{table}

\end{example}
\rvv{Our next example involves a sequence of degenerate nonlinear MPC problems.}
\example[Model Predictive Control]\label{section:MPC}
We now consider the nonlinear system of \cite{zuliani2024convergence} over a finite horizon \rvv{$T=200$} under the feedback action of an MPC controller
\begin{align}
\begin{split}    
x_{t+1} & = f(x_t,u_t,\theta) \\
u_{t} & =\operatorname*{MPC}(x_t,\theta),
\end{split}\label{eq:mpc_nonlinear_dynamics}
\end{align}
where $x_t=(x_t^1,x_t^2)\in\R^2$, with $x_0=(3,0)$, \rvv{$u_t=(u_t^1,u_t^2)\in\R^2$ collects the inputs of two redundant actuators,} and
\begin{align*}
\scalebox{0.9}{$f(x_t,u_t,\theta) :=\begin{bmatrix}
x_t^1 + 0.4x_t^2\\
0.56x_t^2+0.1x_t^1x_t^2+0.4\rvv{(u_t^1+u_t^2)}+\theta x_t^1\exp(-x_t^1)
\end{bmatrix}.$}
\end{align*}
The MPC controller is given by $\operatorname*{MPC}(x_t,\theta)=u_{0|t}^*$ where
\rvv{\begin{align*}
(x_{\cdot|t}^*,u_{\cdot|t}^*) \in &~\operatorname*{argmin}_{x_{\cdot|t},u_{\cdot|t}} \quad \sum_{k=0}^{N} x_{k|t}^{\top}Qx_{k|t}\\
& ~~ \text{s.t.} \quad x_{k+1|t}=f(x_{k|t},u_{k|t},\theta),~ k \in \Z_{[0,N-1]},\\
& ~~ \hphantom{\text{s.t.}} \quad |u_{k|t}^i| \leq 0.475,~i=1,2, ~ k \in \Z_{[0,N-1]}, \\
& ~~ \hphantom{\text{s.t.}} \quad |u_{k|t}^1+u_{k|t}^2| \leq 0.95, ~ k \in \Z_{[0,N-1]}, \\
& ~~ \hphantom{\text{s.t.}} \quad |x_{k|t}^2| \leq 2,~~ x_{0|t}=x_t, ~~ k \in \Z_{[0,N]},
\end{align*}}
with $N=20$, and $Q=\operatorname*{diag}(10^{-2},1)$. 
We denote with $x(\theta):=(x_0,\dots,x_T)$ and $u(\theta):=(u_0,\dots,u_{T-1})$ a pair of closed-loop trajectories obtained by solving \cref{eq:mpc_nonlinear_dynamics} with parameter $\theta$.

\rvv{
The constraint on $u_t^1+u_t^2$ is redundant, and, in terms of the total input $u_t^1+u_t^2$, the feasible set and the cost coincide with those of the single-input MPC of \cite{zuliani2024convergence} with input bound $0.95$ (chosen so that the input saturates along the closed loop). 
At unsaturated stages, moving input from one actuator to the other does not change the cost, yielding flat directions ($\mathcal{P}\neq\{0\}$), so that the SSOSC fails; at saturated stages, both individual bounds and the bound on the sum are active simultaneously, so that the LICQ fails ($\mathcal{U}\neq\{0\}$).
This means that, even if the closed-loop trajectory $x(\theta)$ is identical to that of the single-input MPC, every MPC instance is degenerate by construction.
Nevertheless, since the closed loop only depends on the total input, the closed-loop sensitivity $\partial x(\theta)/\partial\theta$ is unique.}

\begin{figure}
\centering
\begin{tikzpicture}[spy using outlines={rectangle, magnification=1.6, size=1.2cm, connect spies}]

\pgfplotsset{every axis plot/.append style={mark repeat=10}}
\pgfplotsset{numericalA/.style={line4, mark size=0.65, solid, ultra thick, newOrange}}   
\pgfplotsset{numericalB/.style={line4, mark size=0.65, solid, ultra thick, newBlue}}     
\pgfplotsset{ourmethod/.style={
    very thick,
    black,
    dash pattern=on 5.25pt off 1.75pt on 1.75pt off 1.75pt,
}}
\pgfplotsset{dimP/.style={very thick, const plot mark left, purple12546141}}
\pgfplotsset{dimU/.style={very thick, const plot mark left, newRed}}

\pgfplotsset{fig3common/.style={
    tick align=outside,
    tick pos=left,
    xmin=-5, xmax=201,
    xtick distance=50,
    xtick style={color=black},
    ytick style={color=black},
    x grid style={darkgray176},
    y grid style={darkgray176},
    xmajorgrids,
    ymajorgrids,
    enlarge y limits=0.08,
    legend cell align={left},
    legend style={
        fill opacity=0.8,
        draw opacity=1,
        text opacity=1,
        draw=lightgray204,
        font=\small,
        row sep=-4pt,
    },
    tick label style={font=\footnotesize},
    label style={font=\small},
    width=8.5cm, height=3.5cm,
}}

\begin{groupplot}[group style={group size=1 by 2, vertical sep=10pt}, ylabel style={at={(-0.1,0.5)}}]

\nextgroupplot[
fig3common,
xticklabels={},
ylabel={\small $\partial x_t/\partial\theta$},
legend columns=3,
legend style={at={(0.5,1.03)}, anchor=south, /tikz/every even column/.append style={column sep=4pt}},
]
\addplot [numericalA] table [x=t, y=dx1_fd] {\figthreetable};
\addlegendentry{FD, $\partial x^1$}
\addplot [numericalB] table [x=t, y=dx2_fd] {\figthreetable};
\addlegendentry{FD, $\partial x^2$}
\addplot [ourmethod] table [x=t, y=dx1_proj] {\figthreetable};
\addlegendentry{Our method}
\addplot [ourmethod, forget plot] table [x=t, y=dx2_proj] {\figthreetable};
\coordinate (spypoint) at (axis cs:45,1.5);
\coordinate (spyviewer) at (rel axis cs:0.98,0.95);

\nextgroupplot[
fig3common,
xlabel={Time step},
ylabel={\small Dimension},
yticklabel style={/pgf/number format/.cd, fixed, precision=0},
legend columns=2,
legend style={at={(0.98,0.85)}, anchor=north east, /tikz/every even column/.append style={column sep=4pt}},
]
\addplot [dimP] table [x=t, y=dim_P] {\figthreedegtable};
\addlegendentry{$\dim\mathcal{P}$}
\addplot [dimU] table [x=t, y=dim_U] {\figthreedegtable};
\addlegendentry{$\dim\mathcal{U}$}
\addplot [draw=none, forget plot] coordinates {(0,0) (0,2)};
\end{groupplot}

\spy [black, line width = 0.5pt] on (spypoint) in node[anchor=north east] at (spyviewer);

\end{tikzpicture}
\caption{\rvv{Top:} derivative of the closed-loop state trajectory computed with finite differences (\rvv{solid}) and our method (\rvv{dash-dotted black}). \rvv{Bottom: dimension of $\mathcal{P}$ and $\mathcal{U}$ at each MPC instance along the closed loop.}} \label{fig:fig3}
\vspace{-0.45cm}
\end{figure}

We are interested in computing the sensitivity $\partial x (\theta)/ \partial \theta$ of the closed-loop state trajectory $x(\theta)$, which can be obtained through the following propagation scheme
\begin{multline*}
\frac{\partial x_{t+1}(\theta)}{\partial \theta} = \frac{\partial f(x_t,u_t,\theta)}{\partial x_t} \frac{\partial x_t(\theta)}{\rvv{\partial} \theta} + \frac{\partial f(x_t,u_t,\theta)}{\partial \theta} \\ ~~ + \, \frac{\partial f(x_t,u_t,\theta)}{\partial u_t} \left[ \frac{\partial \operatorname*{MPC}(x_t,\theta)}{\partial x_t} \frac{\partial x_t(\theta)}{\rvv{\partial} \theta} + \frac{\partial \operatorname*{MPC}(x_t,\theta)}{\partial \theta} \right],
\end{multline*}
initialized with $\partial x_0(\theta)/\partial \theta = 0$, where the sensitivities $\partial \operatorname*{MPC}(x_t,\theta) / \partial x_t$ and $\partial \operatorname*{MPC}(x_t,\theta) / \partial \theta$ are computed with our method. This greatly extends the procedure we studied in \cite{zuliani2023bp}, where the MPC policy was restricted to a strongly convex quadratic program.

\cref{fig:fig3} \rvv{displays} the closed-loop state derivatives obtained with our method (choosing \rvv{$\rho_1=\rho_2=10^{-4}$}) and the ones obtained with finite differences, highlighting their similarity\rvv{, as well as the dimension of $\mathcal{P}$ and $\mathcal{U}$ along the closed loop, confirming that every MPC instance is degenerate}.
\rvv{We do not report a sweep over $\rho_1$ and $\rho_2$ as the results are insensitive, in this example, to the choice of these parameters.}

\rvv{\begin{remark}[An argument against proximal anchoring]
To overcome the loss of SSOSC, one can add a proximal term to the cost of \probone, for example $\frac{\rho}{2}\|x-\bar x\|^2$, i.e., using the nominal solution as anchor, and solve the resulting system in \cref{eq:derivative} in a least squares sense (the left-hand side may continue to be rank deficient when LICQ is lost).
However, the proximal term acts on all primal variables, including those along which the solution map of \probone is well defined, inducing an $\mathcal{O}(\rho)$ bias in the derivative.
Our method, does not suffer from this, as it acts exclusively on $\mathcal{P}$ and $\mathcal{U}$.
To quantify this effect, we computed the relative error on the closed-loop sensitivity for different values of $\rho$, and reported the results in \cref{tab:prox_ls}.
The table confirms that proximal-anchoring introduces a bias (roughly linear in $\rho$), whereas our method does not.
\begin{table}
\centering
\begin{filecontents*}{Figures/Data/rho_sweep.dat}
% Over-actuated MPC example, m = 2, T = 200, rho_1 = rho_2 = 0.0001, FD step = 1e-06
% Generated by 2_nlp_mpc.py
% proj_rho1: sweep rho_1 (rho_2 fixed), proj_rho2: sweep rho_2 (rho_1 fixed), old: rho*I regularization, prox: proximal anchoring + least squares (eps = rho)
% nan = solve failed
rho proj_rho1_err_fd proj_rho1_err_ref proj_rho1_cos proj_rho1_ms proj_rho2_err_fd proj_rho2_err_ref proj_rho2_cos proj_rho2_ms old_err_fd old_err_ref old_cos old_ms prox_err_fd prox_err_ref prox_cos prox_ms lstsq_err_fd lstsq_err_ref lstsq_cos fd_err_ref
1.0000000000e-08 4.4529264110e-06 5.5219889417e-10 1.0000000000e+00 6.7721066690e+01 4.4529264114e-06 5.5219845976e-10 1.0000000000e+00 6.0583232330e+01 4.0134681039e-06 1.7177860036e-06 1.0000000000e+00 2.8529768000e+00 4.2249824283e-06 4.5901770318e-07 1.0000000000e+00 3.2079346650e+00 4.4529264118e-06 5.5219826365e-10 1.0000000000e+00 4.4533403401e-06
1.0000000000e-07 4.4529264110e-06 5.5219890161e-10 1.0000000000e+00 7.4021274740e+01 4.4529264114e-06 5.5219845976e-10 1.0000000000e+00 5.8454678145e+01 1.5685691825e-05 1.7072263242e-05 9.9999999988e-01 2.1337439650e+00 4.3574329705e-06 4.5876112221e-06 9.9999999999e-01 3.5427379150e+00 4.4529264118e-06 5.5219826365e-10 1.0000000000e+00 4.4533403401e-06
1.0000000000e-06 4.4529264110e-06 5.5219886540e-10 1.0000000000e+00 6.8379298570e+01 4.4529264114e-06 5.5219845976e-10 1.0000000000e+00 6.5044462920e+01 1.6868071653e-04 1.7054000277e-04 9.9999998816e-01 2.0825301400e+00 4.3646610342e-05 4.5869923236e-05 9.9999999926e-01 3.5931552350e+00 4.4529264118e-06 5.5219826365e-10 1.0000000000e+00 4.4533403401e-06
1.0000000000e-05 4.4529264113e-06 5.5219853136e-10 1.0000000000e+00 6.9403276760e+01 4.4529264114e-06 5.5219845976e-10 1.0000000000e+00 6.0699102540e+01 1.6964593477e-03 1.6983597243e-03 9.9999881794e-01 2.0216417100e+00 4.5596143058e-04 4.5833087657e-04 9.9999992415e-01 3.6056686750e+00 4.4529264118e-06 5.5219826365e-10 1.0000000000e+00 4.4533403401e-06
1.0000000000e-04 4.4529264114e-06 5.5219845976e-10 1.0000000000e+00 7.4859917295e+01 4.4529264114e-06 5.5219845976e-10 1.0000000000e+00 6.1042931970e+01 1.6325517360e-02 1.6327406481e-02 9.9989016765e-01 1.9921482750e+00 4.5447101099e-03 4.5470872691e-03 9.9999249507e-01 3.5764684850e+00 4.4529264118e-06 5.5219826365e-10 1.0000000000e+00 4.4533403401e-06
1.0000000000e-03 4.4529264110e-06 5.5219887412e-10 1.0000000000e+00 6.7718422895e+01 4.4529264114e-06 5.5219845976e-10 1.0000000000e+00 5.9086547185e+01 1.1967741584e-01 1.1967921765e-01 9.9391901608e-01 1.9819597100e+00 4.2176558081e-02 4.2178880373e-02 9.9934241720e-01 3.5752508650e+00 4.4529264118e-06 5.5219826365e-10 1.0000000000e+00 4.4533403401e-06
1.0000000000e-02 4.4529264112e-06 5.5219870558e-10 1.0000000000e+00 4.9062160260e+01 4.4529264114e-06 5.5219845976e-10 1.0000000000e+00 4.9857890295e+01 3.7740477688e-01 3.7740642529e-01 9.3496657606e-01 1.9497860750e+00 2.5425277994e-01 2.5425480515e-01 9.7347451428e-01 3.5555938200e+00 4.4529264118e-06 5.5219826365e-10 1.0000000000e+00 4.4533403401e-06
1.0000000000e-01 4.4529264111e-06 5.5219873281e-10 1.0000000000e+00 5.4292930940e+01 4.4529264114e-06 5.5219845976e-10 1.0000000000e+00 5.1459785395e+01 3.7955975978e+07 3.7955848610e+07 1.6667391386e-02 1.9427698350e+00 5.9780313827e-01 5.9780464227e-01 8.1439573352e-01 3.5610827750e+00 4.4529264118e-06 5.5219826365e-10 1.0000000000e+00 4.4533403401e-06
1.0000000000e+00 4.4529264109e-06 5.5219899212e-10 1.0000000000e+00 6.0932333915e+01 4.4529264114e-06 5.5219845976e-10 1.0000000000e+00 5.6781218770e+01 9.5241817843e+16 9.5241498243e+16 8.8312576292e-03 1.9113102750e+00 7.5768187351e-01 7.5768308482e-01 6.5507371896e-01 3.5556986550e+00 4.4529264118e-06 5.5219826365e-10 1.0000000000e+00 4.4533403401e-06
\end{filecontents*}
\pgfplotstabletypeset[
  comment chars={\%},
  every column/.style={column type=c},
  columns={rho, proj_rho1_err_ref, prox_err_ref, proj_rho1_err_fd, prox_err_fd},
  row predicate/.code={\ifodd#1\relax\pgfplotstableuserowfalse\fi},
  err/.style={sci, sci 10e, sci zerofill, precision=1},
  columns/proj_rho1_err_ref/.style={err, column name=Ours},
  columns/prox_err_ref/.style={err, column name=Prox.+LS},
  columns/proj_rho1_err_fd/.style={err, column name=Ours},
  columns/prox_err_fd/.style={err, column name=Prox.+LS},
  columns/rho/.style={
    column name=$\rho$,
    assign cell content/.code={%
      \pgfmathfloatparsenumber{##1}%
      \pgfmathfloatgetexponent{\pgfmathresult}\rhoexp
      \ifnum\rhoexp=0 \def\rhocell{$1$}%
      \else \edef\rhocell{\noexpand$10^{\the\rhoexp}\noexpand$}\fi
      \pgfkeyslet{/pgfplots/table/@cell content}\rhocell
    },
  },
  every head row/.style={
    before row={\toprule
      & \multicolumn{2}{c}{vs.\ reference} & \multicolumn{2}{c}{vs.\ FD} \\},
    after row=\midrule,
  },
  every last row/.style={after row=\bottomrule},
]{Figures/Data/rho_sweep.dat}
\caption{Relative error of $\partial x(\theta)/\partial\theta$ with respect
to the single-actuator reference and to finite differences.}
\label{tab:prox_ls}
\end{table}
\end{remark}}
\rvv{Our last example, taken from \cite{frey2025differentiable}, is a 2-dimensional problem with multiple local minimizers. We focus in particular on a local minimizer and study the surrogate derivative around a $\theta$ where the minimizer vanishes.}
\rvv{\begin{example}[{\hspace{1sp}\cite[Example A.6]{frey2025differentiable}}]
Consider the following example, taken from \cite{frey2025differentiable}
\begin{align}\label{eq:prob_acados}
\begin{split}
\operatorname*{minimize}_{x} & \quad x^4-x^2-\theta x \\ \text{subject to} & ~~ -0.75 \leq x \leq 0.75.
\end{split}
\end{align}
For $\theta$ in a neighborhood of $0$, the problem admits two distinct local minimizers in the interior of the feasible set.
This can be verified by observing that the gradient $\nabla_x f_\theta(x)= 4x^3 - 2x - \theta$ of the objective function is zero whenever $4x^3 - 2x = \theta$.
We can therefore plot $4x^3 - 2x$ and $\theta$ and identify their intersections (see \cref{fig:example_acados}).
\begin{figure}[ht]
\centering
\begin{tikzpicture}
\begin{axis}[
  width=\columnwidth, height=0.6\columnwidth,
  xmin=-0.85, xmax=0.85, ymin=-0.7, ymax=0.7,
  axis lines=middle, axis on top,
  xlabel={$x$}, ylabel={$\theta$},
  xlabel style={anchor=west}, ylabel style={anchor=south},
  xtick={-0.75,-0.4082,0.4082,0.75},
  xticklabels={$-0.75$,$-1/\sqrt{6}$,$1/\sqrt{6}$,$0.75$},
  ytick={-2}, yticklabels={-},
  tick label style={font=\footnotesize},
  legend style={font=\footnotesize, at={(0.5,-0.08)}, anchor=north, draw=none},
  legend columns=3, legend cell align=left,
]
\fill[gray!15] (axis cs:-0.85,-0.7) rectangle (axis cs:-0.75,0.7);
\fill[gray!15] (axis cs:0.75,-0.7) rectangle (axis cs:0.85,0.7);
\addplot[very thick, newBlue, domain=-0.85:0.85, samples=200] {4*x^3 - 2*x};
\addlegendentry{$4x^3-2x$}
\addplot[very thick, newRed, domain=-0.85:0.85] {0.1};
\addlegendentry{$\theta=0.1$}
\addplot[dashed, thick, newRed, domain=-0.85:0.85] {0.5443};
\addlegendentry{$\theta=\tilde\theta$}
\addplot[only marks, mark=*, mark size=1.8pt] coordinates {(-0.6806,0.1) (0.7309,0.1)};
\addlegendentry{local minimizers}
\addplot[only marks, mark=o, mark size=1.8pt, thick] coordinates {(-0.0503,0.1)};
\addlegendentry{local maximizer}
\addplot[only marks, mark=diamond*, mark size=2.5pt, black] coordinates {(-0.4082,0.5443)};
\addlegendentry{fold $(\tilde x,\tilde\theta)$}
\end{axis}
\end{tikzpicture}
\caption{\rvv{Critical points of \cref{eq:prob_acados} for different values of $\theta$.}}
\label{fig:example_acados}
\end{figure}
In \cref{fig:example_acados} we can additionally discriminate local maximizers and local minimizers by observing the slope of the derivative (where a positive slope indicates a local minimizer). 
Focusing on the left local minimizer, we can see that upon increasing $\theta$, the intersection with the red line moves up, forcing the minimizer and the maximizer to converge towards the fold located at $-1/\sqrt{6}$ (reaching it exactly when $\theta=\tilde{\theta}=2 \sqrt{6}/9$).

For $\theta < \tilde{\theta}$ the left minimizer is nondegenerate, meaning that $\Pi_{\mathcal{P}} = \Pi_{\mathcal{U}} = 0$ and our method returns the exact derivative $1/\nabla_{xx}^2 f_\theta(x) = 1 / (12 x^2 - 2)$, which grows to infinity as $\theta \to \tilde{\theta}$.
The first-order model is valid only up to $\tilde \theta$, where the local minimizer disappears.
When $\theta=\tilde{\theta}$ and $x=\tilde{x}$, the point $\tilde x$ is a KKT point with $\nabla_{xx}^2 f_{\tilde\theta}(\tilde x) = 0$.
It is not a local minimizer, so \cref{ass:SONC} does not hold; however, \eqref{eq:SONC_P1} is still satisfied, so our method can be applied and we have $\mathcal{P} = \mathbb{R}$ and $\mathcal{U} = \{0\}$.
The regularized problem amounts to minimizing $f_\theta(x) + \frac{\rho_1}{2}(x - \tilde x)^2$, which admits $\tilde x$ as a strict local minimizer, and the surrogate derivative is $1/\rho_1$.
Near the fold, upon neglecting cubic terms, we have $\nabla_x f_{\tilde\theta + \varepsilon}(x) \approx -\varepsilon - c(x - \tilde x)^2$, where $c=2\sqrt{6}$, meaning that the solution of $P_3(\tilde\theta + \varepsilon, \bar\xi, \rho)$ is approximately
\begin{align*}
x^\rho(\varepsilon) \approx \tilde x + \frac{\rho_1 - \sqrt{\rho_1^2 - 4c\varepsilon}}{2c},
\end{align*}
which exists only for $\varepsilon \le \rho_1^2/(4c)$.
The regularization therefore shifts the fold by $O(\rho_1^2)$ but does not remove it, showing that the restriction $|\varepsilon| \le \beta\rho_1^2$ in \cref{prop:subopt} (found in the Appendix) is tight up to the constant.
\end{example}}
%
\section{Conclusion}
In this paper\rvv{,} we propose a principled way to obtain sensitivities of solution maps of nonlinear optimization problems without requiring the strong second-order sufficient conditions of optimality or the linear independence constraint qualification. This technical result greatly relaxes the set of assumptions commonly found in the literature and can be used in future work to provide convergence guarantees to policy optimization schemes using nonlinear MPC controllers under less restrictive assumptions than the ones currently found in the literature, see, for example, \cite{davis2020stochastic}. We showcase the effectiveness of our method in both a trajectory optimization and an MPC example. 
Future work will focus on developing specialized software to implement the method more efficiently, \rvv{potentially with a tailored interior-point QP solver}, and utilizing the proposed algorithm within a policy optimization framework.
%
\bibliographystyle{IEEEtran}
\bibliography{ref.bib}
\appendix
%
\section{Analysis}\label{section:analysis}
\rvv{\cref{theorem:bound} proves that the regularization does not affect the derivative when \probone is nondegenerate.
We now study its effect in the degenerate case, deriving guidelines for the choice of $\rho$.
We work under \cref{ass:SONC} and SCS, and restrict attention to the active constraints $c_{\theta,\mathcal{I}}(x):=\mathrm{col}(g_{\theta,\mathcal{I}}(x),h_\theta(x))$.
For simplicity we will assume $\rho_1=\rho_2$.
All multipliers, as well as $\bar\mu$ and $\Pi_{\mathcal{U}}$, are restricted accordingly, with inactive multipliers set to zero.
For a fixed $\bar{\phi}$, let
\begin{align*}
\Gamma:=\nabla_x c_{\theta,\mathcal{I}}(\bar x),\quad
H:=\nabla_{xx}^2\bar{\mathcal{L}}_\theta(\bar\phi),\quad
A=\begin{bmatrix}H&\Gamma^\top\\\Gamma&0\end{bmatrix}.
\end{align*}
By \cref{eq:degenerate_subspaces:U}, $\mathcal{U}=\operatorname{null}\Gamma^\top$, which is trivial exactly when the LICQ holds.
Under the SCS, the critical cone reduces to $\mathcal{C}=\operatorname{null}\Gamma$, so that \cref{eq:SONC_P1} is equivalent to $Z^\top HZ\succeq0$ for $Z$ spanning $\operatorname{null}\Gamma$.

\begin{lemma}\label{lemma:null_A}
Under \cref{ass:SONC} and SCS, $\mathcal{P}=\{p:\Gamma p=0,\ p^\top Hp=0\}$; for every $p\in\mathcal{P}$ there is a unique $w_p\in\mathcal{U}^\perp=\operatorname{range}\Gamma$ with $Hp=\Gamma^\top w_p$, and the map $p\mapsto w_p$ is linear. 
Moreover
\begin{equation}\label{eq:null_A}
\operatorname{null}A=\{(0,u):u\in\mathcal{U}\}\oplus\{(p,-w_p):p\in\mathcal{P}\},
\end{equation}
so that every $v\in\operatorname{null}A$ is uniquely of the form $v=(p,u-w_p)$ with $p\in\mathcal{P}$ and $u\in\mathcal{U}$.
\end{lemma}
\begin{proof}
Let $Z$ span $\operatorname{null}\Gamma$, so that $Z^\top HZ\succeq0$ by \cref{eq:SONC_P1}, and let $p=Zq$. 
Then $p^\top Hp$ vanishes if and only if $Z^\top HZq=0$, because a positive semidefinite form vanishes exactly on its null space. 
Hence $\{p:\Gamma p=0,\ p^\top Hp=0\}=Z\operatorname{null}(Z^\top HZ)$ is a subspace, and thus coincides with its span $\mathcal{P}$. 
For $p\in\mathcal{P}$ we get $Z^\top Hp=0$, meaning that $Hp$ is orthogonal to $\operatorname{null}\Gamma$ and thus belongs to $\operatorname{range}\Gamma^\top$.
Any two vectors $w,w'$ with $\Gamma^\top w=\Gamma^\top w'=Hp$ differ by an element of $\operatorname{null}\Gamma^\top=\mathcal{U}$, so $w_p$ is unique once we require $w_p\in\mathcal{U}^\perp$; linearity of $p\mapsto w_p$ follows from uniqueness.

For the inclusion $\supseteq$ in \cref{eq:null_A}, note that $A(0,u)=(\Gamma^\top u,0)=0$ for $u\in\mathcal{U}$ and $A(p,-w_p)=(Hp-\Gamma^\top w_p,\Gamma p)=0$ for $p\in\mathcal{P}$. 
Conversely, let $A(x,\psi)=0$.
Then $\Gamma x=0$ and $Hx=-\Gamma^\top\psi$, so $x^\top Hx=-(\Gamma x)^\top\psi=0$ and therefore $x\in\mathcal{P}$; consequently $\Gamma^\top(\psi+w_x)=0$, i.e., $\psi=-w_x+u$ for some $u\in\mathcal{U}$, and $(x,\psi)=(x,-w_x)+(0,u)$. 
The sum is direct because $(p,-w_p)=(0,u)$ forces $p=0$ and hence $u=0$.
\end{proof}

Formal differentiation of the KKT conditions of \probone at $\bar\phi$ in a direction $d\in\mathbb{R}^{n_\theta}$ gives
\begin{equation}\label{eq:sensitivity_d}
A\begin{bmatrix}dx\\d\psi\end{bmatrix}=-r:=-\big(\nabla^2_{\theta x}\bar{\mathcal{L}}_\theta(\bar\phi)d,\; \nabla_\theta c_{\theta,\mathcal{I}}(\bar x)d\big),
\end{equation}
which is solvable only under the following two families of conditions.

\begin{lemma}\label{lemma:solvability}
Under \cref{ass:SONC} and SCS, the system $Av=-r$ admits a solution if and only if
\begin{equation}\label{eq:solvability}
\Pi_{\mathcal{U}}r_2=0,\quad p^\top r_1-w_p^\top r_2=0\quad\text{for all }p\in\mathcal{P}.
\end{equation}
\end{lemma}
\begin{proof}
By Fredholm's alternative, since $A$ is symmetric, the system has a solution if and only if $y^\top r=0$ for all $y\in\operatorname{null}A$. 
By \cref{lemma:null_A} it suffices to test the two families of generators, which gives $u^\top r_2=0$ for $u\in\mathcal{U}$ and $p^\top r_1-w_p^\top r_2=0$ for $p\in\mathcal{P}$.
\end{proof}

Since $p\mapsto p^\top r_1-w_p^\top r_2$ is linear, let $\chi_{\mathcal{P}}(r)\in\mathcal{P}$ be its Riesz representative, i.e., the unique element of $\mathcal{P}$ with
\begin{equation}\label{eq:riesz}
\chi_{\mathcal{P}}(r)^\top p=p^\top r_1-w_p^\top r_2\quad \text{for all }p\in\mathcal{P},
\end{equation}
and let $\chi_{\mathcal{U}}(r):=\Pi_{\mathcal{U}}r_2\in\mathcal{U}$.
Both are linear in $r$.
We measure the degeneracy of \probone along $d$ through the \emph{degeneracy certificates} $\chi_{\mathcal{P}}(r)$ and $\chi_{\mathcal{U}}(r)$. 
By \cref{lemma:solvability}, they vanish if and only if \cref{eq:sensitivity_d} is solvable.
This is true in particular if $A$ is invertible (i.e., \probone is nondegenerate).

The regularized counterpart of $A$ is $\tilde A:=\tilde A^\rho_\theta(\bar\xi)=A+R$, with $R:=\mathrm{diag}(\rho_1\Pi_{\mathcal{P}},-\rho_2\Pi_{\mathcal{U}})$, which is nonsingular for all $\rho_1,\rho_2>0$ by \cref{prop:regularized_problem_is_cont_diff}.

\begin{lemma}[Exact $\rho$-decomposition]\label{lemma:decomposition}
Let $\tilde r:=\big(r_1-\chi_{\mathcal{P}}(r),\,r_2-\chi_{\mathcal{U}}(r)\big)$ and let $\tilde{v}=(\tilde{v}_x,\tilde{v}_\psi)$ be the unique vector satisfying
\begin{equation}\label{eq:v_star}
A\tilde{v}=-\tilde r,\quad \Pi_{\mathcal{P}}\tilde{v}_x=0,\quad \Pi_{\mathcal{U}}\tilde{v}_\psi=0.
\end{equation}
Then, for all $\rho>0$, the solution $v=(v_x,v_\psi)$ of $\tilde A v=\!-r$ is
\begin{equation}\label{eq:decomposition}
v_x=\tilde{v}_x-\frac{1}{\rho_1}\chi_{\mathcal{P}}(r),\quad v_\psi=\tilde{v}_\psi+\frac{1}{\rho_1}w_{\chi_{\mathcal{P}}(r)} +\frac{1}{\rho_2}\chi_{\mathcal{U}}(r).
\end{equation}
In particular, $\rho_1\Pi_{\mathcal{P}}v_x=-\chi_{\mathcal{P}}(r)$ and
$\rho_2\Pi_{\mathcal{U}}v_\psi=\chi_{\mathcal{U}}(r)$.
\end{lemma}
\begin{proof}
We first prove that $\tilde r\perp\operatorname{null}A$.
First, taking $(0,u)$ with $u\in\mathcal{U}$, we have $u^\top(r_2-\chi_{\mathcal{U}}(r))=u^\top\Pi_{\mathcal{U}}r_2-u^\top\Pi_{\mathcal{U}}r_2=0$.
Next, taking $(p,-w_p)$ with $p\in\mathcal{P}$, we have $p^\top(r_1-\chi_{\mathcal{P}}(r))-w_p^\top(r_2-\chi_{\mathcal{U}}(r))=p^\top r_1-w_p^\top r_2-\chi_{\mathcal{P}}(r)^\top p+w_p^\top\chi_{\mathcal{U}}(r)=0$, using \cref{eq:riesz} and $w_p\in\mathcal{U}^\perp$, $\chi_{\mathcal{U}}(r)\in\mathcal{U}$. 
Hence $\tilde r\in\operatorname{range}A$ and the solution set of $Av=-\tilde r$ is $v^0+\operatorname{null}A$ for any particular solution $v^0$. 
By \cref{lemma:null_A}, adding $(p,u-w_p)\in\operatorname{null}A$ to $v^0$ changes $\Pi_{\mathcal{P}}v_x$ by $p$ and $\Pi_{\mathcal{U}}v_\psi$ by $u$, so exactly one element of the solution set satisfies the projection conditions in \cref{eq:v_star}.

Next, let $v$ be defined by \cref{eq:decomposition} and put
\begin{align*}
n:=v-\tilde{v}=-\frac{1}{\rho_1}\big(\chi_{\mathcal{P}}(r),-w_{\chi_{\mathcal{P}}(r)}\big) +\frac{1}{\rho_2}\big(0,\chi_{\mathcal{U}}(r)\big).
\end{align*}
Since $\chi_{\mathcal{P}}(r)\in\mathcal{P}$ and $\chi_{\mathcal{U}}(r)\in\mathcal{U}$, $n\in\operatorname{null}A$ by \cref{lemma:null_A}, and therefore $Av=A\tilde{v}=-\tilde r$.
Moreover, using $\Pi_{\mathcal{P}}\tilde{v}_x=0$, $\Pi_{\mathcal{U}}\tilde{v}_\psi=0$ and $w_{\chi_{\mathcal{P}}(r)}\in\mathcal{U}^\perp$,
\begin{align*}
Rv=\big(\rho_1\Pi_{\mathcal{P}}v_x,\,-\rho_2\Pi_{\mathcal{U}}v_\psi\big) =-\big(\chi_{\mathcal{P}}(r),\chi_{\mathcal{U}}(r)\big),
\end{align*}
so $\tilde A v=-\tilde r-(\chi_{\mathcal{P}}(r),\chi_{\mathcal{U}}(r))=-r$.
Uniqueness follows from the nonsingularity of $\tilde A$.
\end{proof}

Given $d$, let $dx^\rho=\bar\nabla^\rho_\theta x(\theta;\bar\xi)\,d$ and $d\psi^\rho=\bar\nabla^\rho_\theta \psi(\theta;\bar\xi)\,d$, so that $\tilde A(dx^\rho,d\psi^\rho)=-r$ by \cref{prop:regularized_problem_is_cont_diff}.

\begin{corollary}\label{cor:certificates}
For all $\rho_1,\rho_2>0$,
\begin{equation}\label{eq:chi_computable}
\rho_1\Pi_{\mathcal{P}}\,dx^\rho=-\chi_{\mathcal{P}}(r),\quad \rho_2\Pi_{\mathcal{U}}\,d\psi^\rho=\chi_{\mathcal{U}}(r),
\end{equation}
and $(dx^\rho,d\psi^\rho)$ stays bounded as $\rho\to0$ if and only if \cref{eq:sensitivity_d} is solvable.
\end{corollary}

For $\varepsilon$ near $0$, let $\phi^\rho(\varepsilon)=(x^\rho(\varepsilon),\psi^\rho(\varepsilon))$ denote the primal-dual solution of $P_3(\theta+\varepsilon d,\bar\xi,\rho)$.
By \cref{prop:regularized_problem_is_cont_diff}, $\phi^\rho(\varepsilon)$ is continuously differentiable, with $\phi^\rho(0)=(\bar x,\bar\mu)$ and derivative $(dx^\rho,d\psi^\rho)$ at $\varepsilon=0$.

\begin{proposition}\label{prop:subopt}
For all $\varepsilon$ near $0$,
\begin{subequations}\label{eq:backward}
\begin{align}
\nabla_x\bar{\mathcal{L}}_{\theta+\varepsilon d}\big(\phi^\rho(\varepsilon)\big)
&=-\rho_1\Pi_{\mathcal{P}}\big(x^\rho(\varepsilon)-\bar x\big),
\label{eq:backward_stat}\\
c_{\theta+\varepsilon d,\mathcal{I}}\big(x^\rho(\varepsilon)\big)
&=\rho_2\Pi_{\mathcal{U}}\big(\psi^\rho(\varepsilon)-\bar\mu\big).
\label{eq:backward_feas}
\end{align}
\end{subequations}
If, in addition, $f_\theta,g_\theta,h_\theta$ are three times continuously differentiable in $(x,\theta)$ and $\rho_2=\rho_1$, then there exist $\beta,c>0$, independent of $\rho_1$, such that for all $\rho_1\in(0,1]$ and $|\varepsilon|\le\beta\rho_1^2$,
\begin{subequations}\label{eq:asymptotics}
\begin{align}
\big\|\nabla_x\bar{\mathcal{L}}_{\theta+\varepsilon d}\big(\phi^\rho(\varepsilon)\big)
-\varepsilon\,\chi_{\mathcal{P}}(r)\big\|
&\le c\,\varepsilon^2/\rho_1^2,\label{eq:res_stat}\\
\big\|c_{\theta+\varepsilon d,\mathcal{I}}\big(x^\rho(\varepsilon)\big)
-\varepsilon\,\chi_{\mathcal{U}}(r)\big\|
&\le c\,\varepsilon^2/\rho_1^2,\label{eq:res_feas}\\
\big\|\phi^\rho(\varepsilon)-(\bar x,\bar\mu)
-\varepsilon(dx^\rho,d\psi^\rho)\big\|
&\le c\,\varepsilon^2/\rho_1^3.\label{eq:taylor}
\end{align}
\end{subequations}
\end{proposition}

\begin{proof}
\emph{Step 1: reduction to a nonlinear system.}
Define the KKT map of \probone restricted to the active constraints,
$F(\phi,\varepsilon):=\big(\nabla_x\bar{\mathcal{L}}_{\theta+\varepsilon d}(x,\psi),\,c_{\theta+\varepsilon d,\mathcal{I}}(x)\big)$,
and its regularized counterpart
$F_\rho(\phi,\varepsilon):=F(\phi,\varepsilon)+\big(\rho_1\Pi_{\mathcal{P}}(x-\bar x),\,-\rho_2\Pi_{\mathcal{U}}(\psi-\bar\mu)\big)$.
Using $\mu=\Pi_{\mathcal{U}}\psi$ and $\delta=\psi_{\mathrm{in}}$ from
\cref{eq:KKT_3}, and the fact that every element of $\mathcal{U}$ vanishes on
the inactive components, it is straightforward to verify that, near
$(\phi^\rho(0),0)$, the KKT points of $P_3(\theta+\varepsilon d,\bar\xi,\rho)$
are exactly the zeros of $F_\rho$ satisfying
\begin{equation}\label{eq:sign_cond}
\psi_{\mathrm{in},i}>0,\ \forall i\in\mathcal{I}, \qquad
g_{\theta+\varepsilon d,i}(x)<0,\ \forall i\notin\mathcal{I},
\end{equation}
with $z_{\mathcal{I}}=0$ and $z_i=g_{\theta+\varepsilon d,i}(x)$ for
$i\notin\mathcal{I}$.
By the SCS, \cref{eq:sign_cond} holds at $(\phi^\rho(0),0)$, and hence along
$\phi^\rho(\varepsilon)$ for small $\varepsilon$ by continuity. This proves
\cref{eq:backward}.
Moreover, $F_\rho(\phi^\rho(0),0)=0$,
$D_\phi F_\rho(\phi^\rho(0),0)=\tilde A$, and
$\partial_\varepsilon F(\phi^\rho(0),0)=r$. Since $F_\rho-F$ is affine in
$\phi$ and independent of $\varepsilon$, $F$ and $F_\rho$ have the same
second derivatives.

\emph{Step 2: $\|\tilde A^{-1}\|\le c_1/\rho_1$.}
Let $v:=-\tilde A^{-1}r$ for some $r$. 
By \cref{lemma:decomposition}, the vectors $\tilde v$, $\chi_{\mathcal{P}}(r)$, $w_{\chi_{\mathcal{P}}(r)}$ and $\chi_{\mathcal{U}}(r)$ depend linearly on $r$ and not on $\rho_1$, so they are bounded by a multiple of $\|r\|$.
Hence $\|v\|\le c(1+2/\rho_1)\|r\|\le 3c\|r\|/\rho_1$ for $\rho_1\in(0,1]$, i.e., $\|\tilde A^{-1}\|\le c_1/\rho_1$ with $c_1:=3c$ independent of $\rho_1$.

\emph{Step 3: the solution branch on $|\varepsilon|\le\beta\rho_1^2$.}
Fix a compact convex neighborhood $\mathcal{N}_0$ of $(\phi^\rho(0),0)$, which does not depend on $\rho$.
On $\mathcal{N}_0$, $F$ is twice continuously differentiable, so there are constants $c_2,c_3$ such that $\|D^2F\|\le c_2$ and $\|\partial_\varepsilon F\|\le c_3$.
Since $D_\phi F_\rho-D_\phi F$ is constant by Step 1, Taylor's theorem gives
\begin{align*}
\|D_\phi F_\rho(\phi,\varepsilon)-\tilde A\| \leq c_2\big(\|\phi-\phi^\rho(0)\|+|\varepsilon|\big).
\end{align*}
By the Banach perturbation lemma, $D_\phi F_\rho(\phi,\varepsilon)$ therefore remains invertible for $(\phi,\varepsilon)$ close enough to $(\phi^\rho(0),0)$.
Quantitatively, on the set $\mathcal{N}_\rho:=\{(\phi,\varepsilon)\in\mathcal{N}_0: \|\phi-\phi^\rho(0)\|+|\varepsilon|\le\rho_1/(2c_1c_2)\}$, Step 2 gives $\|\tilde A^{-1}\|\,\|D_\phi F_\rho(\phi,\varepsilon)-\tilde A\|\le 1/2$, and hence $\|D_\phi F_\rho(\phi,\varepsilon)^{-1}\|\le 2\|\tilde A^{-1}\|\le2c_1/\rho_1$.
By the implicit function theorem, the branch through $(\phi^\rho(0),0)$ can be continued as long as it remains in $\mathcal{N}_\rho$, with $\phi^{\rho\prime}(\varepsilon)=-D_\phi F_\rho^{-1}\partial_\varepsilon F$.
Hence $\|\phi^{\rho\prime}\|\le2c_1c_3/\rho_1$ and $\|\phi^\rho(\varepsilon)-\phi^\rho(0)\|\le2c_1c_3|\varepsilon|/\rho_1$ along the branch.
Consequently, the branch cannot leave $\mathcal{N}_\rho$ while $|\varepsilon|(1+2c_1c_3)\le\rho_1^2/(2c_1c_2)$.

Along the branch, $\|\phi^\rho(\varepsilon)-\phi^\rho(0)\|\leq 2c_1c_3\beta\rho_1\le2c_1c_3\beta$.
Choosing $\beta$ small enough (and independent of $\rho_1$) therefore also preserves \cref{eq:sign_cond}, ensuring that the branch consists of KKT points of \probthree.

\emph{Step 4: estimates.}
Differentiating $F_\rho(\phi^\rho(\varepsilon),\varepsilon)=0$ twice gives $D_\phi F_\rho\,\phi^{\rho\prime\prime}=-q$, where $q:=D^2_{\phi\phi}F[\phi^{\rho\prime},\phi^{\rho\prime}] +2D^2_{\phi\varepsilon}F[\phi^{\rho\prime}]+D^2_{\varepsilon\varepsilon}F$.
By Step 3, $\|q\|\le c_2(1+2c_1c_3/\rho_1)^2=O(1/\rho_1^2)$, and hence $\|\phi^{\rho\prime\prime}\|\le(2c_1/\rho_1)\|q\|=O(1/\rho_1^3)$ on $|\varepsilon|\le\beta\rho_1^2$. 
Taylor's theorem then gives \cref{eq:taylor}.

For \cref{eq:res_stat} and \cref{eq:res_feas}, \cref{eq:backward_stat}, \cref{eq:backward_feas}, and \cref{eq:chi_computable} give
\begin{align*}
\nabla_x\bar{\mathcal{L}}_{\theta+\varepsilon d}(\phi^\rho(\varepsilon))
-\varepsilon\chi_{\mathcal{P}}(r) &=-\rho_1\Pi_{\mathcal{P}}\big(x^\rho(\varepsilon)-\bar x-\varepsilon\,dx^\rho\big),\\
c_{\theta+\varepsilon d,\mathcal{I}}(x^\rho(\varepsilon))-\varepsilon\chi_{\mathcal{U}}(r) &=\rho_1\,\Pi_{\mathcal{U}}\big(\psi^\rho(\varepsilon)-\bar\mu-\varepsilon\,d\psi^\rho\big),
\end{align*}
and both norms are at most $\rho_1\cdot c\,\varepsilon^2/\rho_1^3$ by \cref{eq:taylor}.
\end{proof}

By \cref{eq:backward}, $x^\rho(\varepsilon)$ satisfies exactly the KKT conditions of the perturbed problem
\begin{align}\label{eq:perturbed}
\begin{split}
\operatorname*{minimize}_x & \quad f_{\theta+\varepsilon d}(x) +\tfrac{\rho_1}{2}\|x-\bar x\|^2_{\Pi_{\mathcal{P}}} \\
\text{subject to} & \quad c_{\theta+\varepsilon d,\mathcal{I}}(x)=\rho_2\Pi_{\mathcal{U}}(\psi^\rho(\varepsilon)-\bar\mu).
\end{split}
\end{align}
The surrogate derivative is therefore the exact derivative of the solution of a problem that differs from $P_1(\theta+\varepsilon d)$ in two ways: a proximal term along $\mathcal{P}$, and a shift of the active constraints along $\mathcal{U}$.
The resulting violation of the KKT conditions of $P_1(\theta+\varepsilon d)$ is first-order in $\varepsilon$ by \cref{eq:res_stat,eq:res_feas}, driven by the degeneracy certificates $\chi_{\mathcal{P}}(r)$ and $\chi_{\mathcal{U}}(r)$, which do not depend on $\rho$ and are available from the computed derivative through \cref{cor:certificates}.

\emph{Role of $\rho$.} For $|\varepsilon|\le\beta\rho_1^2$, \cref{eq:backward_stat,eq:backward_feas,eq:res_stat,eq:res_feas} give
\begin{align*}
\|\nabla_x\bar{\mathcal L}_{\theta+\varepsilon d}(\phi^\rho(\varepsilon))\| & = |\varepsilon|\,\|\chi_{\mathcal P}(r)\| + O(\varepsilon^2/\rho_1^2),\\
\|c_{\theta+\varepsilon d,\mathcal I}(x^\rho(\varepsilon))\| & = |\varepsilon|\,\|\chi_{\mathcal U}(r)\| + O(\varepsilon^2/\rho_1^2).
\end{align*}
Hence, to first order, the violation of the KKT conditions of $P_1(\theta+\varepsilon d)$ is determined by
the degeneracy certificates and cannot be reduced by the choice of $\rho$. What $\rho$ controls is the
first-order model itself:
\begin{itemize}
\item \emph{Small $\rho$.} The components $\Pi_{\mathcal P}dx^\rho=-\chi_{\mathcal P}(r)/\rho_1$ and
$\Pi_{\mathcal U}d\psi^\rho=\chi_{\mathcal U}(r)/\rho_2$ grow, the model has relative error $O(\varepsilon/\rho_1^2)$
whenever $\chi_{\mathcal P}(r)\neq0$ or $\chi_{\mathcal U}(r)\neq0$, and its accuracy is only guaranteed for
$|\varepsilon|\lesssim\rho_1^2$, a restriction that is tight by Example~4. The condition number of $\tilde A$ also grows.
\item \emph{Large $\rho$.} The model is accurate over a larger range of perturbations, but it describes the
solution of (25), which is more strongly anchored at $\bar x$ along $\mathcal P$.
\end{itemize}
If $\chi_{\mathcal P}(r)=\chi_{\mathcal U}(r)=0$, the surrogate derivative does not depend on $\rho$ (Lemma~4),
and $\rho$ only affects conditioning. Otherwise, $\rho_1$ should be chosen so that the perturbations of
interest (e.g., the step size of a policy optimization scheme) lie within $|\varepsilon|\le\beta\rho_1^2$.

\subsection{Robustness to proximal anchor updates}\label{sec:robustness}

A typical use of the surrogate derivative is within a policy optimization scheme.
There, the parameter is updated iteratively as $\theta_{k+1}=\theta_k+\alpha_k d_k$, with step size $\alpha_k>0$ and a search direction $d_k$ computed from $\bar\nabla^\rho_\theta x(\theta_k;\bar\xi_k)$.
At every iteration, \probone is solved at $\theta_{k+1}$, typically warm-started at $\bar\xi_k$, and the new solution $\bar\xi_{k+1}$ becomes the proximal anchor of the regularized problem.
The anchor therefore moves along the iterations, and it is only known up to the accuracy of the solver.
We now show that, if the step size is small enough, moving the proximal anchor does not destroy the properties established in \cref{lemma:regularized_problem_is_diff_ae,prop:regularized_problem_is_cont_diff}.
Given $\hat\xi=(\hat x,\hat z,\hat\mu)$, let $P_3(\tilde\theta,\hat\xi,\rho)$ denote \longprobthree with $\theta$, $\bar x$ and $\bar\mu$ replaced by $\tilde\theta$, $\hat x$ and $\hat\mu$.
In this problem, $\Pi_{\mathcal{P}}$ and $\Pi_{\mathcal{U}}$ are kept equal to the projectors at $\bar\xi$.

\begin{proposition}\label{prop:robustness}
Let $\bar\phi=(\bar x,\bar\lambda,\bar\nu)\in\Phi(\theta)$ satisfy \cref{ass:SONC} and the SCS, let $\rho_1,\rho_2>0$, and let $\eta$ be the KKT point of \longprobthree given by \cref{lemma:regularized_problem_is_diff_ae}.
Then there exists $\bar\epsilon>0$ such that, for every $\tilde\theta$ and $\hat\xi$ with $\|\tilde\theta-\theta\|\le\bar\epsilon$ and $\|\hat\xi-\bar\xi\|\le\bar\epsilon$, problem $P_3(\tilde\theta,\hat\xi,\rho)$ admits a primal-dual solution $\eta(\tilde\theta;\hat\xi)$, with $\eta(\theta;\bar\xi)=\eta$, that
\begin{enumerate}
\item is a locally unique KKT point of $P_3(\tilde\theta,\hat\xi,\rho)$ and has the same active set as $\eta$;
\item satisfies the LICQ, the SCS, and the SSOSC, so that its primal part is a strict local minimizer of $P_3(\tilde\theta,\hat\xi,\rho)$;
\item is continuously differentiable in $(\tilde\theta,\hat\xi)$; in particular, its derivative with respect to $\tilde\theta$ depends continuously on $\hat\xi$.
\end{enumerate}
\end{proposition}
\begin{proof}
The cost and constraints of $P_3(\tilde\theta,\hat\xi,\rho)$ are twice continuously differentiable in $\xi$, and their gradients with respect to $\xi$ are continuously differentiable in $(\tilde\theta,\hat x,\hat\mu)$, since $\hat x$ and $\hat\mu$ enter affinely and $\Pi_{\mathcal{P}}$, $\Pi_{\mathcal{U}}$ are fixed.
By \cref{lemma:regularized_problem_is_diff_ae}, $\eta$ satisfies the LICQ, the SCS, and the SSOSC for $P_3(\theta,\bar\xi,\rho)$.
Applying \cref{lemma:differentiability} with respect to the parameter $(\tilde\theta,\hat x,\hat\mu)$ yields a neighborhood $N$ of $(\theta,\bar x,\bar\mu)$ and a continuously differentiable map $(\tilde\theta,\hat x,\hat\mu)\mapsto\eta(\tilde\theta;\hat\xi)$ with $\eta(\theta;\bar\xi)=\eta$, such that $\eta(\tilde\theta;\hat\xi)$ is a primal-dual solution of $P_3(\tilde\theta,\hat\xi,\rho)$ satisfying the LICQ, the SCS, and the SSOSC for all $(\tilde\theta,\hat x,\hat\mu)\in N$.
Choose $\bar\epsilon$ such that $(\tilde\theta,\hat x,\hat\mu)\in N$ whenever $\|\tilde\theta-\theta\|\le\bar\epsilon$ and $\|\hat\xi-\bar\xi\|\le\bar\epsilon$.
This proves 2) and 3).

For 1), note that $z=g_\theta(\bar x)$ and $\delta=\bar\lambda$ at $\eta$.
By the SCS, $z_i<0$ for $i\notin\mathcal{I}$ and $\delta_i>0$ for $i\in\mathcal{I}$.
These inequalities are strict, so they hold on a neighborhood $M$ of $\eta$.
By complementarity, every KKT point of $P_3(\tilde\theta,\hat\xi,\rho)$ in $M$ has active set $\mathcal{I}$ and therefore solves the KKT system in active-set form, i.e., with $z_i\delta_i=0$ replaced by $z_i=0$ for $i\in\mathcal{I}$ and $\delta_i=0$ for $i\notin\mathcal{I}$.
The Jacobian of this system at $\eta$ is nonsingular, as in the proof of \cref{prop:regularized_problem_is_cont_diff}.
By the implicit function theorem, after possibly shrinking $M$ and $\bar\epsilon$, the system has a unique solution in $M$ for each $(\tilde\theta,\hat\xi)$.
Since $\eta(\cdot\,;\cdot)$ is continuous, $\eta(\tilde\theta;\hat\xi)\in M$, so $\eta(\tilde\theta;\hat\xi)$ is this solution.
\end{proof}}
%
\subsection{Proof of \texorpdfstring{\hyperref[lemma:regularized_problem_is_diff_ae]{Proposition 1}}{Proposition 1}}\label{section:proof_prop_1}
\rvv{\begin{proof}
Since $\Pi_\mathcal{U}(\bar{\mu}-\Pi_\mathcal{U}\bar{\mu})=0$, it is easy to verify from \cref{eq:KKT_3} that $\eta$ is a KKT point of \longprobthree, with multiplier $\delta=\bar{\lambda}$ for the constraint $z\leq 0$.
The active constraints at $\eta$ are $\tilde{c}_\theta^\rho(\xi;\bar{\mu}):=c_\theta(x,z)+\rho_2\Pi_\mathcal{U}(\bar{\mu}-\mu)=0$ and $z_i=0$ for $i\in\mathcal{I}$, whose Jacobian with respect to $\xi=(x,z,\mu)$ is
\begin{align}
A:=\begin{bmatrix} B & -\rho_2\Pi_\mathcal{U} \\ \begin{bmatrix} 0 & I_\mathcal{I} \end{bmatrix} & 0 \end{bmatrix}, \qquad B:=\begin{bmatrix} \nabla_x g_\theta(\bar{x}) & -I\\ \nabla_x h_\theta(\bar{x}) & 0 \end{bmatrix}, \label{eq:LICQ_condition}
\end{align}
where $I_\mathcal{I}$ denotes the rows of $I_{n_\text{in}}$ indexed by $\mathcal{I}$. To prove the LICQ, let $y=(y_\text{in},y_\text{eq})\in\R^{n_\mu}$ and $y_\mathcal{I}\in\R^{|\mathcal{I}|}$ satisfy $A^\top(y,y_\mathcal{I})=0$, that is
\begin{subequations}
\begin{align}
\nabla_x g_\theta(\bar{x})^\top y_\text{in} + \nabla_x h_\theta(\bar{x})^\top y_\text{eq} & = 0, \label{eq:LICQ_x}\\
-y_\text{in} + I_\mathcal{I}^\top y_\mathcal{I} & = 0, \label{eq:LICQ_z}\\
-\rho_2\Pi_\mathcal{U}y & = 0.\label{eq:LICQ_mu}
\end{align}
\end{subequations}
By \cref{eq:LICQ_z}, $y_\text{in}=I_\mathcal{I}^\top y_\mathcal{I}$, meaning that $y_{\text{in},i}=0$ for $i\notin\mathcal{I}$.
From \cref{eq:LICQ_x}, we see that $y\in\mathcal{U}$ by definition of $\mathcal{U}$, so that $y=\Pi_\mathcal{U}y$.
Then \cref{eq:LICQ_mu} gives $y=\Pi_\mathcal{U}y=0$, and consequently $y_\mathcal{I}=0$ by \cref{eq:LICQ_z}. Therefore $A$ is full row rank for any $\rho_1,\rho_2>0$, proving LICQ.

Next, by \cref{eq:P3_lagrangian}, the Hessian of the Lagrangian at $\eta$ is
\begin{align}
\nabla_{\xi\xi}^2 \mathcal{L}_\theta^\rho(\eta;\bar{\xi}) = \operatorname*{diag}\big(\nabla_{xx}^2 \bar{\mathcal{L}}_\theta(\bar{\phi})+\rho_1\Pi_\mathcal{P},~0,~\rho_2 I\big). \label{eq:KKT_mat_3_part_1}
\end{align}
To show the SOSC, we need to verify that the quadratic form
\begin{align}
v_x^\top\big(\nabla_{xx}^2 \bar{\mathcal{L}}_\theta(\bar{\phi})+\rho_1\Pi_\mathcal{P}\big)v_x + \rho_2\|v_\mu\|^2 \label{eq:SOSC_form}
\end{align}
is positive for all $v=(v_x,v_z,v_\mu)\neq 0$ such that
\begin{align}
v_{z,i}=0, \forall i\in\mathcal{I}_+,~~ v_{z,i}\leq 0, \forall i\in\mathcal{I}_0,~~ B v_{xz} = \rho_2\Pi_\mathcal{U}v_\mu,\label{eq:SOSC_constraint}
\end{align}
where $v_{xz}:=(v_x,v_z)$. 
The component $(I-\Pi_\mathcal{U})v_\mu$ does not enter \cref{eq:SOSC_constraint} and can only increase \cref{eq:SOSC_form}, so we may equivalently assume $\Pi_\mathcal{U}v_\mu=v_\mu$.
Eliminating $v_\mu$ via \cref{eq:SOSC_constraint}, \cref{eq:SOSC_form} becomes $v_{xz}^\top Q(\rho_2)v_{xz}$, where
\begin{align*}
Q(\rho_2) = \begin{bmatrix} \nabla_{xx}^2 \bar{\mathcal{L}}_\theta(\bar{\phi})+\rho_1\Pi_\mathcal{P} & 0 \\ 0 & 0 \end{bmatrix} + \frac{1}{\rho_2}B^\top B.
\end{align*}
If $Bv_{xz}=0$ and $v_z$ satisfies \cref{eq:SOSC_constraint}, then $v_z=\nabla_x g_\theta(\bar{x})v_x$ by the definition of $B$, and therefore $v_x\in\mathcal{C}$.
By \cref{eq:SONC_P1} and the definition of $\mathcal{P}$, the first term of $Q(\rho_2)$ is then positive for all $v_{xz}$ for which $Bv_{xz}=0$ unless $v_{xz}=0$. 
Finsler's lemma \cite[Lemma 4.28]{ruszczynski2011nonlinear}, which can be extended to the closed cone defined on $v_z$ in \cref{eq:SOSC_constraint} by a standard compactness argument, guarantees existence of some $\bar{\rho}_2>0$ such that $Q(\rho_2)$ is positive on this cone for all $\rho_2\in(0,\bar{\rho}_2]$. 
This proves the SOSC, and thus that $\xi^*$ is a strict local minimizer.

Finally, under SCS we have $\mathcal{I}_0=\emptyset$ and $\delta_i=\bar{\lambda}_i>0$ for $i\in\mathcal{I}$, so the SCS holds for \longprobthree and the SOSC coincides with the SSOSC.
Moreover, for any $y\in\mathcal{U}$ we have $y^\top Bv_{xz}=-y_\text{in}^\top v_z=0$ by definition of $\mathcal{U}$ and because $y_{\text{in},i}=0$ for $i\notin\mathcal{I}$ and conversely $v_{z,i}=0$ for $i\in\mathcal{I}$.
This means that $Bv_{xz}\in \mathcal{U}^\perp$, and since $\rho_2 \Pi_\mathcal{U}v_\mu \in \mathcal{U}$, we must have that both sides lie in $\mathcal{U}^{\perp} \cap \mathcal{U}=\{0\}$.
Hence, positivity of \cref{eq:SOSC_form} follows without appealing to Finsler's lemma, for all $\rho_1,\rho_2>0$.
\end{proof}}
%
\subsection{Proof of \texorpdfstring{\hyperref[prop:regularized_problem_is_cont_diff]{Proposition 2}}{Proposition 2}}\label{section:proof_prop_2}
\begin{proof}
The continuous differentiability claim follows from \cref{lemma:differentiability} since \rvv{SCS,} SSOSC\rvv{,} and LICQ hold by \cref{lemma:regularized_problem_is_diff_ae} \rvv{for all $\rho_1,\rho_2>0$}. To prove the second part of the Proposition we apply the implicit function theorem to the KKT conditions in \cref{eq:KKT_3}. We have
\begin{align*}
\nabla^2_{\psi x}\mathcal{L}_\theta^\rho(\eta;\bar{\xi}) & = \begin{bmatrix} \nabla_x g_\theta(\bar{x})^{\top} & \nabla_x h_\theta(\bar{x})^{\top} \end{bmatrix},\\
\nabla_{\psi z}^2 \mathcal{L}_\theta^\rho(\eta;\bar{\xi}) & = \rvv{\begin{bmatrix} -I & 0 \end{bmatrix}},\\
\nabla_{\psi \mu}^2 \mathcal{L}_\theta^\rho(\eta;\bar{\xi}) & = - \rvv{\rho_2 \Pi_\mathcal{U}}.
\end{align*}
Combining with \cref{eq:KKT_mat_3_part_1}, the Jacobian of the left-hand side of \cref{eq:KKT_3} with respect to $\eta$ is
\rvv{\begin{align*}
A:=\scalebox{0.75}{$\begin{bmatrix}
    \nabla_{xx}^2 \bar{\mathcal{L}}_\theta(\bar{\phi})\!+\!\rho_1 \Pi_\mathcal{P} & 0 & 0 & [\nabla_x g_\theta(\bar{x})^{\top} ~ \nabla_x h_\theta(\bar{x})^{\top}] & 0 \\
    0 & 0 & 0 & [-I ~ 0] & I \\
    0 & 0 & \rho_2 I & -\rho_2 \Pi_\mathcal{U} & 0\\
    \begin{bmatrix} \nabla_x g_\theta(\bar{x}) \\ \nabla_x h_\theta(\bar{x}) \end{bmatrix} &  \begin{bmatrix} -I \\ 0 \end{bmatrix} & -\rho_2 \Pi_\mathcal{U} & 0 & 0 \\
    0 & \operatorname*{diag}(\delta) & 0 & 0 & \operatorname*{diag}(z)
    \end{bmatrix}$}
\end{align*}}
and the Jacobian with respect to $\theta$ is \rvv{$b=(b_x,0,0,b_\psi,0)$ where $b_x = - \nabla_{\theta x}^2 \bar{\mathcal{L}}_\theta (\bar{\phi})$, and $b_\psi = - \nabla_\theta c_\theta(\bar{x},0)$.}
\rvv{From the third row of the system $Av=b$ we have $v_\mu=\Pi_\mathcal{U}v_\psi$, and therefore $-\rho_2\Pi_\mathcal{U}v_\mu=-\rho_2\Pi_\mathcal{U}v_\psi$ where $v=(v_x,v_z,v_\mu,v_\psi,v_\delta)$.} So we can replace the system $Av=b$ with the smaller \rvv{$Cw=d$, where $w=(v_x,v_z,v_{\psi_1},v_{\psi_2},v_\delta)$, $d=(b_x,0,b_{\psi_1},b_{\psi_2},0)$ and}
\rvv{\begin{align*}
C = \scalebox{0.87}{$\begin{bmatrix}
    \nabla_{xx}^2 \bar{\mathcal{L}}_\theta(\bar{\phi})\!+\!\rho_1 \Pi_\mathcal{P} & 0 & \nabla_x g_\theta(\bar{x})^{\top} & \nabla_x h_\theta(\bar{x})^{\top} & 0 \\
    0 & 0 & -I & 0 & I \\
    \nabla_x g_\theta(\bar{x}) & -I & -\rho_2 \Pi_\mathcal{U}^\text{in,in} & -\rho_2 \Pi_\mathcal{U}^\text{in,eq} & 0 \\
    \nabla_x h_\theta(\bar{x}) & 0 & -\rho_2 \Pi_\mathcal{U}^\text{eq,in} & -\rho_2 \Pi_\mathcal{U}^\text{eq,eq} & 0 \\
    0 & \operatorname*{diag}(\delta) & 0 & 0 & \operatorname*{diag}(z)
    \end{bmatrix}$}
\end{align*}}
\rvv{Solving the second row of $Cw=d$ for $v_\delta$ gives $v_\delta=v_{\psi_1}$, and replacing this into the fifth row yields $\operatorname*{diag}(\delta)v_z=-\operatorname*{diag}(z)v_{\psi_{1}}$.} 
\rvv{Under SCS, $\delta_i>0=z_i$ for $i\in\mathcal{I}$ and $\delta_i=0>z_i$ for $i\notin\mathcal{I}$, meaning that $v_{z,i}=0$ for $i\in\mathcal{I}$ and $v_{\psi_1,i}=0$ for $i\notin\mathcal{I}$. Moreover, every $u\in\mathcal{U}$ satisfies $u_{\text{in},i}=0$ for $i\notin\mathcal{I}$, so the rows and columns of $\Pi_\mathcal{U}$ associated with $\psi_{\text{in},i}$, $i\notin\mathcal{I}$, vanish. Hence the rows of the third block row of $Cw=d$ indexed by $i\notin\mathcal{I}$ only determine $v_{z,i}$ and decouple from the remaining equations.}
\rvv{Using this information, we can remove $v_\delta$, $v_z$, and $v_{\psi_1,i}$, $i\notin\mathcal{I}$, from the system $Cw=d$ and consider the reduced system $D \omega = e$, where $\omega=(v_x,v_{\psi_1,\mathcal{I}},v_{\psi_2})$, $e=(b_x,b_{\psi_1,\mathcal{I}},b_{\psi_2})$, and $D=\tilde{A}_\theta^\rho(\bar{\xi})$ is as in \cref{eq:approx_derivative}.} 
This concludes the proof.
\end{proof}
%
\subsection{Proof of \texorpdfstring{\hyperref[prop:directional_derivative]{Proposition 3}}{Proposition 3}}\label{section:proof_prop_3}
\rvv{\begin{proof}
By \cref{lemma:regularized_problem_is_diff_ae} and \cref{cor:SSOSC}, $\eta=(\bar{x},\bar{z},\Pi_\mathcal{U}\bar{\mu},\bar{\mu},\bar{\lambda})$ is a KKT point of \longprobthree satisfying the LICQ and, for all $\rho_2\in(0,\bar{\rho}_2]$, the SSOSC. Since \longprobthree satisfies the LICQ and the SSOSC at $\eta$, its primal-dual solution map is locally unique, Lipschitz continuous, and directionally differentiable at $\theta$, and its directional derivative $v=(v_x,v_z,v_\mu)$ in the direction $d$ is the unique solution of the quadratic program
\begin{align*}
\operatorname*{minimize}_{v} & \quad \frac{1}{2}v^\top\nabla_{\xi\xi}^2 \mathcal{L}_\theta^\rho(\eta;\bar{\xi})v + v^\top\nabla^2_{\theta\xi}\mathcal{L}_\theta^\rho(\eta;\bar{\xi})d\\
\text{subject to} & \quad Bv_{xz}-\rho_2\Pi_\mathcal{U}v_\mu+\nabla_\theta c_\theta(\bar{x},0)d=0,\\
& \quad v_{z,i}=0,~ i\in\mathcal{I}_+, \qquad v_{z,i}\leq 0,~ i\in\mathcal{I}_0,
\end{align*}
see \cite{Jittorntrum1984}. 
Here $\nabla^2_{\theta\xi}\mathcal{L}_\theta^\rho(\eta;\bar{\xi})d=(\nabla^2_{\theta x}\bar{\mathcal{L}}_\theta(\bar{\phi})d,0,0)$, because the proximal terms and the relaxation do not depend on $\theta$, and the objective contains the term $\frac{\rho_2}{2}\|v_\mu\|^2=\frac{\rho_2}{2}\|\Pi_\mathcal{U}v_\mu\|^2+\frac{\rho_2}{2}\|(I-\Pi_\mathcal{U})v_\mu\|^2$. 
Since $(I-\Pi_\mathcal{U})v_\mu$ does not appear in the constraints, it vanishes at the solution, and we may set $w:=\Pi_\mathcal{U}v_\mu\in\mathcal{U}$. 
Next, the rows of the equality constraint associated with $g_{\theta,i}$ read $v_{z,i}=\nabla_x g_{\theta,i}(\bar{x})v_x+\nabla_\theta g_{\theta,i}(\bar{x})d-\rho_2w_{\text{in},i}$. For $i\notin\mathcal{I}$ they only determine $v_{z,i}$, which does not appear elsewhere, while for $i\in\mathcal{I}_+$ and $i\in\mathcal{I}_0$ the conditions $v_{z,i}=0$ and $v_{z,i}\leq 0$ become the first two constraints in \cref{eq:directional_derivative_regularized}. 
Finally, the rows associated with $h_\theta$ give the third constraint. 
Eliminating $v_z$ and $(I-\Pi_\mathcal{U})v_\mu$ yields \cref{eq:directional_derivative_regularized}, whose solution is unique because the two programs are equivalent.
\end{proof}}
%
\subsection{Proof of \texorpdfstring{\hyperref[theorem:bound]{Theorem 1}}{Theorem 1}}\label{section:proof_thm_1}
\rvv{\begin{proof}
By LICQ the gradients of the active constraints are linearly independent, so $\mathcal{U}=\{0\}$. Moreover, the SSOSC for \probone requires $p^\top\nabla_{xx}^2 \bar{\mathcal{L}}_\theta(\bar{\phi}) p>0$ for all $p\in\mathcal{K}_+\setminus\{0\}$, so that \cref{ass:strong_SONC} holds and $\mathcal{P}\subseteq\mathcal{P}_+=\{0\}$.

We first consider the general case. The SSOSC for \longprobthree holds at $\eta$ for all $\rho_1,\rho_2>0$: since $\mathcal{U}=\{0\}$, every $v$ in the proof of \cref{cor:SSOSC} satisfies $Bv_{xz}=0$ and $v_{z,i}=0$ for $i\in\mathcal{I}_+$, so that $v_x\in\mathcal{K}_+$ and $v_z=\nabla_x g_\theta(\bar{x})v_x$, and the quadratic form $v_x^\top\nabla_{xx}^2 \bar{\mathcal{L}}_\theta(\bar{\phi})v_x+\rho_2\|v_\mu\|^2$ is positive unless $v=0$. Since the restriction on $\rho_2$ was only used to establish the SSOSC, the conclusions of \cref{prop:directional_derivative} hold for all $\rho_1,\rho_2>0$. With $\Pi_{\mathcal{P}_+}=0$ and $w\in\mathcal{U}=\{0\}$, \cref{eq:directional_derivative_regularized} reduces to
\begin{align*}
\operatorname*{minimize}_{v_x} & \quad \frac{1}{2}v_x^\top\nabla_{xx}^2 \bar{\mathcal{L}}_\theta(\bar{\phi})v_x + v_x^\top \nabla^2_{\theta x}\bar{\mathcal{L}}_\theta(\bar{\phi})d\\
\text{subject to} & \quad \nabla_x g_{\theta,i}(\bar{x})v_x + \nabla_\theta g_{\theta,i}(\bar{x})d = 0, \quad i\in\mathcal{I}_+,\\
& \quad \nabla_x g_{\theta,i}(\bar{x})v_x + \nabla_\theta g_{\theta,i}(\bar{x})d \leq 0, \quad i\in\mathcal{I}_0,\\
& \quad \nabla_x h_\theta(\bar{x})v_x + \nabla_\theta h_\theta(\bar{x})d = 0,
\end{align*}
which does not depend on $\rho$. Since \probone satisfies the LICQ and the SSOSC at $\phi$, by \cite{Jittorntrum1984} this is exactly the quadratic program whose unique solution is $x'(\theta;d)$. Hence $\bar{x}^{\rho\,\prime}(\theta;d)=x'(\theta;d)$.

Suppose now that, additionally, the SCS holds. Then $\mathcal{I}_0=\emptyset$ and $\mathcal{C}=\mathcal{K}_+$ is the subspace on which the SSOSC requires $p^\top\nabla_{xx}^2 \bar{\mathcal{L}}_\theta(\bar{\phi}) p>0$ for $p\neq 0$, so $\mathcal{P}=\{0\}$. Therefore $\Pi_\mathcal{P}=0$ and $\Pi_\mathcal{U}=0$, and \cref{eq:approx_derivative} coincides with the linear system \cref{eq:derivative} of \cref{lemma:differentiability}, whose solution is unique. Hence the first $n_x$ rows of $V$ coincide with $\nabla_\theta x(\theta)$.
\end{proof}}
\end{document}